\documentclass{amsart}
\usepackage{amsthm, amssymb, amsmath, a4wide}
\usepackage[all]{xy} 
\usepackage{paralist} 
\usepackage{graphicx} 
\usepackage[sans]{dsfont} 
\usepackage{varioref} 
\usepackage{tikz}
\usepackage{pgf}
\usetikzlibrary{decorations.pathreplacing,calc}
\usepackage{xifthen}

\usepackage[
]{caption}
\usepackage[list=true]{subcaption}

\usepackage[backend=bibtex, style=alphabetic, backref=true]{biblatex}
\bibliography{bibi}

\usepackage{hyperref}

\title{
Algebraicity of normal analytic compactifications of $\cc^2$ with one irreducible curve at infinity
}
\author{Pinaki Mondal}
\address{College of The Bahamas}
\email{pinakio@gmail.com}

\subjclass[2010]{Primary 32J05, 14J26; Secondary 14E05, 14E15}

\DeclareMathOperator\lt{LT}
\makeatletter

\newcommand{\Rmnum}[1]{\expandafter\@slowromancap\romannumeral #1@}
\makeatother

\DeclareMathOperator\gr{gr}
 
\DeclareMathOperator\ord{ord}

\DeclareMathOperator\proj{Proj}

\DeclareMathOperator\spec{Spec}

\newcommand{\scrB}{\ensuremath{\mathcal{B}}}

\newcommand{\scrE}{\ensuremath{\mathcal{E}}}
\newcommand{\scrF}{\ensuremath{\mathcal{F}}}
\newcommand{\scrG}{\ensuremath{\mathcal{G}}}

\newcommand{\cc}{\ensuremath{\mathbb{C}}}

\newcommand{\pp}{\ensuremath{\mathbb{P}}}
\newcommand{\qq}{\ensuremath{\mathbb{Q}}}
\newcommand{\rr}{\ensuremath{\mathbb{R}}}

\newcommand{\zz}{\ensuremath{\mathbb{Z}}}

\newcommand{\ppp}{\ensuremath{\mathfrak{p}}}

\newcommand{\sheaf}{\ensuremath{\mathcal{O}}}

\newcommand{\im}{\ensuremath{\Rightarrow}}

\newcommand{\dsum}{\ensuremath{\bigoplus}}

\newtheorem{thm}{Theorem}[section]
\newtheorem*{thm*}{Theorem}
\newtheorem{lemma}[thm]{Lemma}
\newtheorem*{lemma*}{Lemma}

\newtheorem{prop}[thm]{Proposition}
\newtheorem*{prop*}{Proposition}
\newtheorem{cor}[thm]{Corollary}
\newtheorem{claim}[thm]{Claim}
\newtheorem*{claim*}{Claim}
\newtheorem{proclaim}{Claim}[thm]
\newtheorem*{conjecture*}{Conjecture}

\theoremstyle{definition} 
\newtheorem{algorithm}[thm]{Algorithm}

\newtheorem*{constrinition*}{Construction-Definition}
\newtheorem{convention}[thm]{Convention}
\newtheorem*{convention*}{Convention}
\newtheorem{defn}[thm]{Definition}
\newtheorem*{defn*}{Definition}

\newtheorem*{definotation*}{Definition-Notation}
\newtheorem{example}[thm]{Example}
\newtheorem*{example*}{Example}

\newtheorem*{fact*}{Fact}

\newtheorem*{facts*}{Facts}

\newtheorem{notation}[thm]{Notation}

\newtheorem*{bold-note*}{Note}

\newtheorem*{problem*}{Problem}
\newtheorem{bold-question}[thm]{Question}
\newtheorem*{bold-question*}{Question}
\newtheorem{rem}[thm]{Remark}
\newtheorem{reminition}[thm]{Remark-Definition}
\newtheorem*{reminition*}{Remark-Definition}
\newtheorem{remexample}[thm]{Remark-Example}
\newtheorem{remexample*}{Remark-Example}
\newtheorem{remtation}[thm]{Remark-Notation}
\newtheorem*{remtation*}{Remark-Notation}

\newtheorem*{remuestion*}{Remark-Question}

\theoremstyle{remark}
\newtheorem*{rem*}{Remark}
\newtheorem*{note*}{Note}
\newtheorem*{notation*}{Notation}
\newtheorem*{question*}{Question}

\newtheorem*{questions*}{Questions}

\newcounter{UnorderedProofTempCtr}
\newcommand{\tempcommand}{}

\newcommand{\xdelta}{\bar X^\delta}

\newcommand{\ldt}{\text{l.d.t.}}

\newcommand{\ld}{\mathfrak{L}}

\newcommand{\dpsx}[2]{{#1} \langle \langle #2 \rangle \rangle }
\newcommand{\dpsxc}{\dpsx{\cc}{x}}

\newcommand{\polsub}[1]{($\text{Polynomial}_{#1}$)}

\newcommand{\Xxy}{\bar X_{(x,y)}}

\theoremstyle{definition} 

\begin{document}

\begin{abstract}
We present an effective criterion to determine if a normal analytic compactification of $\cc^2$ with one irreducible curve at infinity is algebraic or not. As a by product we establish a correspondence between normal algebraic compactifications of $\cc^2$ with one irreducible curve at infinity and algebraic curves contained in $\cc^2$ with one place at infinity. Using our criterion we construct pairs of homeomorphic normal analytic surfaces with minimally elliptic singularities such that one of the surfaces is algebraic and the other is not. Our main technical tool is the sequence of {\em key forms} - a `global' variant of the sequence of {\em key polynomials} introduced by MacLane \cite{maclane-key} to study valuations in the `local' setting - which also extends the notion of {\em approximate roots} of polynomials considered by Abhyankar-Moh \cite{abhya-moh-tschirnhausen}.
\end{abstract}

\maketitle

\section{Introduction} \label{sec-intro}
Algebraic compactifications of $\cc^2$ (i.e.\ compact algebraic surfaces containing $\cc^2$) are in a sense the simplest compact algebraic surfaces. The simplest among these are the {\em primitive compactifications}, i.e.\ those for which the complement of $\cc^2$ (a.k.a.\ the {\em curve at infinity}) is irreducible. It follows from a famous result of Remmert and Van de Ven that up to isomorphism, $\pp^2$ is the only {\em nonsingular} primitive compactification of $\cc^2$. In some sense a more natural category than nonsingular algebraic surfaces is the category of {\em normal} algebraic surfaces\footnote{This is true for example from the perspective of valuation theory: the irreducible components of the curve at infinity of a normal compactification $\bar X$ of $\cc^2$ correspond precisely to the discrete valuations on $\cc[x,y]$ which are {\em centered at infinity} with positive dimensional center on $X$. Therefore $\bar X$ is primitive and normal iff $\bar X$ corresponds to precisely one discrete valuation centered at infinity on $\cc[x,y]$}. In this article we tackle the problem of understanding the simplest normal algebraic compactifications of $\cc^2$:

\begin{bold-question} \label{compactification-question}
What are the normal primitive algebraic compactifications of $\cc^2$? 
\end{bold-question}

We give a complete answer to this question; in particular, we characterize both algebraic and {\em non-algebraic} primitive compactifications of $\cc^2$. Our answer is equivalent to an {\em explicit} criterion for determining {\em algebraicity} of (analytic) contractions of a class of curves: indeed, it follows from well known results of Kodaira, and independently of Morrow, that any normal analytic compactification $\bar X$ of $\cc^2$ is the result of contraction of a (possibly reducible) curve $E$ from a non-singular surface constructed from $\pp^2$ by a sequence of blow-ups. On the other hand, a well known result of Grauert completely and {\em effectively} characterizes all curves on a nonsingular analytic surface which can be {\em analytically} contracted: namely it is necessary and sufficient that the matrix of intersection numbers of the irreducible components of $E$ is negative definite. It follows that the question of understanding algebraicity of analytic compactifications of $\cc^2$ is 
equivalent to the following question:

\begin{bold-question} \label{contract-question}
Let $\pi: Y \to \pp^2$ be a birational morphism of nonsingular complex algebraic surfaces and $L \subseteq \pp^2$ be a line. Assume $\pi$ restricts to an isomorphism on $\pi^{-1}(\pp^2\setminus L)$. Let $E$ be the {\em exceptional divisor} of $\pi$ (i.e.\ $E$ is the union of curves on $Y$ which map to points in $\pp^2$) and $E_1, \ldots, E_N$ be irreducible curves contained in $E$. Let $E'$ be the union of the strict transform $L'$ (on $Y$) of $L$ and all components of $E$ {\em excluding} $E_1, \ldots, E_N$. Assume $E'$ is analytically contractible; let $\pi':Y \to Y'$ be the contraction of $E'$. When is $Y'$ algebraic?
\end{bold-question}
%
Question \ref{compactification-question} is equivalent to the $N=1$ case of Question \ref{contract-question}. We give a complete solution to this case of Question \ref{contract-question} (Theorem \ref{effective-answer}). Our answer is in particular {\em effective}, i.e.\ given a description of $Y$ (e.g.\ if we know a sequence of blow ups which construct $Y$ from $\pp^2$, or if we know precisely the discrete valuation $\nu$ on $\cc(x,y)$ associated to the unique curve on $Y\setminus \cc^2$ which does {\em not} get contracted), our algorithm determines in finite time if the contraction is algebraic. In fact the algorithm is a one-liner: ``Compute the {\em key forms} of $\nu$. $Y'$ is algebraic iff the last key form is a polynomial.'' The only previously known effective criteria for determining the algebraicity of contraction of curves on surfaces was the well-known criteria of Artin \cite{artractability} which states that a normal surface is algebraic if all its singularities are {\em rational}. We refer the reader to \cite{morrow-rossi}, \cite{brenton-algebraicity}, \cite{franco-lascu}, \cite{schroe-traction}, \cite{badescu-contractibility}, \cite{palka-Q1} for other criteria - some of these are more general, but none is effective in the above sense. Moreover, as opposed to Artin's criterion, ours is {\em not} numerical, i.e.\ it is not determined by numerical invariants of the associated singularities. We give an example (in Section \ref{non-example-section}) which shows that in fact there is {\em no topological},  let alone numerical, answer to Question \ref{contract-question} even for $N=1$.\\

As a corollary of our criterion, we establish a new correspondence between normal primitive algebraic compactifications of $\cc^2$ and algebraic curves in $\cc^2$ with {\em one place at infinity}\footnote{Let $C \subseteq \cc^2$ be an algebraic curve, and let $\bar C'$ be the closure of $C$ in $\pp^2$ and $\sigma:  \bar C \to \bar C'$ be the desingularization of $\bar C'$. $C$ has {\em one place at infinity} iff $|\sigma^{-1}(\bar C' \setminus C)| = 1$.} (Theorem \ref{geometric-answer}). Curves with one place at infinity have been extensively studied in affine algebraic geometry (see e.g.\ \cite{abhya-moh-tschirnhausen}, \cite{abhya-moh-line}, \cite{ganong}, \cite{russburger}, \cite{naka-oka}, \cite{suzuki}, \cite{wightwick}), and we believe the connection we found between these and compactifications of $\cc^2$ will be useful for the study of both\footnote{E.g.\ we use this connection in \cite{cpl} to solve completely the main problem studied in \cite{campillo-piltant-lopez-cones-surfaces}.}.  \\

Our main technical tool is the sequence of {\em key forms}, which is a direct analogue of the sequence {\em key polynomials} introduced by MacLane \cite{maclane-key}. The key polynomials were introduced (and have been extensively used - see e.g.\ \cite{moyls}, \cite{favsson-tree}, \cite{vaquie}, \cite{herrera-olalla-spivakovsky}) to study valuations in a {\em local} setting. However, our criterion shows how they retain information about the {\em global geometry} when computed in `global coordinates.'\\

The example in Section \ref{non-example-section} shows that algebraicity of $Y'$ from Question \ref{contract-question} can not be determined only from the (weighted) {\em dual graph} (Definition \ref{dual-defn}) of $E'$. However, at least when $N=1$, it is possible to completely characterize the weighted dual graphs (more precisely, {\em augmented and marked} weighted dual graphs - see Definition \ref{augmented-defn}) which correspond to {\em only algebraic} contractions, those which correspond to {\em only non-algebraic} contractions, and those which correspond to {\em both} types of contractions (Theorem \ref{graph-thm}). The characterization involves two sets of {\em semigroup conditions} \eqref{semigroup-criterion-1} and \eqref{semigroup-criterion-2}. We note that the first set of semigroup conditions \eqref{semigroup-criterion-1} are equivalent to the semigroup conditions that appear in the theory of plane curves with one place at infinity developed in \cite{abhya-moh-tschirnhausen}, \cite{abhyankar-expansion}, \cite{abhyankar-semigroup}, \cite{sathaye-stenerson}.\\

Finally we would like to point that Question \ref{compactification-question} is equivalent to a two dimensional {\em Cousin-type problem at infinity}: let $O_1, \ldots, O_N \in \pp^2\setminus \cc^2$ be {\em points at infinity}. Let $(u_j, v_j)$ be coordinates near $O_j$, $\psi_j(u_j) $ be a {\em Puiseux series} (Definition \ref{meromorphic-defn}) in $u_j$, and $r_j$ be a positive rational number, $1 \leq j \leq N$. 

\begin{bold-question} \label{cousin-question}
Determine if there exists a polynomial $f \in \cc[x,y]$ such that for each analytic branch $C$ of the curve $f = 0$ at infinity, there exists $j$, $1 \leq j \leq N$, such that 
\begin{itemize}
\item $C$ intersects $L_\infty$ at $O_j$,
\item  $C$ has a Puiseux expansion $v_j = \theta(u_j)$ at $O_j$ such that $\ord_{u_j}(\theta - \psi_j) \geq r_j$. 
\end{itemize}
\end{bold-question}

On our way to understand normal primitive compactifications of $\cc^2$, we solve the $N=1$ case of Question \ref{cousin-question} (Theorem \ref{cousin-answer}).

\begin{rem}
We use {\em Puiseux series} in an essential way in this article. However, instead of the usual Puiseux series, from Section \ref{firstground} onward, we almost exclusively work with {\em descending} Puiseux series (a descending Puiseux series in $x$ is simply a meromorphic Puiseux series in $x^{-1}$ - see Definition \ref{dpuiseux}). The choice was enforced on us `naturally' from the context - while key polynomials and Puiseux series are natural tools in the study of valuations in the local setting, when we need to study the relation of valuations corresponding to curves at infinity (on a compactification of $\cc^2$) to global properties of the surface, key forms and descending Puiseux series are sometimes more convenient.
\end{rem}

\subsection{Organization}
We start with an example in Section \ref{non-example-section} to illustrate that the answer to Question \ref{contract-question} can not be numerical or topological. The construction also serves as an example of non-algebraic normal {\em Moishezon surfaces}\footnote{\em{Moishezon surfaces} are analytic surfaces for which the fields of meromorphic functions have transcendence degree 2 over $\cc$} with the `simplest possible' singularities (see Remark \ref{simplest-remark}). In Section \ref{firstground} we recall some background material and in Section \ref{result-section} we state our results. The rest of the article is devoted to the proof of the results of Section \ref{result-section}. In Section \ref{secondground} we recall some more background material needed for the proof; in particular in Section \ref{essential-section} we state the algorithm to compute key forms of a valuation from the associated descending Puiseux series, and illustrate the algorithm via an elaborate example (we note that this algorithm is essentially the same as the algorithm used in \cite{lenny} for a different purpose). In Section \ref{prep-section} we build some tools for dealing with descending Puiseux series and in Section \ref{proof-section} we use these tools to prove the results from Section \ref{result-section}. The appendices contain proof of two lemmas from Section \ref{prep-section} - the proofs were relegated to the appendix essentially because of their length. 

\subsection{Acknowledgements}
This project started during my Ph.D.\ under Professor Pierre Milman to answer some of his questions, and profited enormously from his valuable suggestions and speaking in his seminars. I would like to thank Professor Peter Russell for inviting me to present this result and for his helpful remarks. The idea for presenting the `effective answer' in terms of key polynomials came from a remark of Tommaso de Fernex, and Mikhail Zaidenberg asked  the question of characterization of (non-)algebraic dual graphs during a poster presentation. I thank Mattias Jonsson and Mark Spivakovsky for bearing with my bugging at different stages of the write up, and Leonid Makar-Limanov for crucial encouragement during the (long) rewriting stage and pointing out the equivalence of the algorithm of \cite{lenny} and our Algorithm \ref{key-algorithm}. Finally I have to thank Dmitry Kerner for forcing me to think geometrically by his patient (and relentless) questions. Theorem \ref{graph-thm}, and the example from Section \ref{non-example-section} appeared in \cite{trento}. 


\section{Algebraic and non-algebraic compactifications with homeomorphic singularities} \label{non-example-section}

Let $(u,v)$ be a system of `affine' coordinates near a point $O \in \pp^2$ (`affine' means that both $u=0$ and $v=0$ are lines on $\pp^2$) and $L$ be the line $\{u=0\}$. Let $C_1$ and $C_2$ be curve-germs at $O$ defined respectively by $f_1 := v^5 - u^3$ and $f_2 := (v-u^2)^5 - u^3$. For each $i$, let $Y_i$ be the surface constructed by resolving the singularity of $C_i$ at $O$ and then blowing up $8$ more times the point of intersection of the (successive) strict transform of $C_i$ with the exceptional divisor. Let $E^*_i$ be the {\em last} exceptional curve, and $E'^{(i)}$ be the union of the strict transform $L'_i$ (on $Y_i$) of $L$ and (the strict transforms of) all exceptional curves except $E^*_i$. 

\begin{figure}[htp]

\centering
\subcaptionbox[]{%
    Weighted dual graph of $E'^{(i)}$%
    \label{fig:non-example-a}%
}
[%
    0.48\textwidth 
]%
{%
\begin{tikzpicture}[scale=1, font = \small] 

	\pgfmathsetmacro\factor{1.25}
 	\pgfmathsetmacro\dashedge{2*\factor}	
 	\pgfmathsetmacro\edge{.75*\factor}
 	\pgfmathsetmacro\vedge{.55*\factor}

 	\draw[thick] (-2*\edge,0) -- (\edge,0);
 	\draw[thick] (0,0) -- (0,-2*\vedge);
 	\draw[thick, dashed] (\edge,0) -- (\edge + \dashedge,0);
 	\draw[thick] (\edge + \dashedge,0) -- (2*\edge + \dashedge,0);
 	
 	\fill[black] ( - 2*\edge, 0) circle (3pt);
 	\fill[black] (-\edge, 0) circle (3pt);
 	\fill[black] (0, 0) circle (3pt);
 	\fill[black] (0, -\vedge) circle (3pt);
 	\fill[black] (0, - 2*\vedge) circle (3pt);
 	\fill[black] (\edge, 0) circle (3pt);
 	\fill[black] (\edge + \dashedge, 0) circle (3pt);
 	\fill[black] (2*\edge + \dashedge, 0) circle (3pt);
 	
 	\draw (-2*\edge,0)  node (e0up) [above] {$-1$};
 	\draw (-2*\edge,0 )  node (e0down) [below] {$L'_i$};
 	\draw (-\edge,0 )  node (e2up) [above] {$-3$};
 	\draw (-\edge,0 )  node [below] {$E_2$};
 	\draw (0,0 )  node (e4up) [above] {$-2$};	
 	\draw (0,0 )  node [below left] {$E_4$};
 	\draw (0,-\vedge )  node (down1) [right] {$-2$};
 	\draw (0,-\vedge )  node [left] {$E_3$};
 	\draw (0, -2*\vedge)  node (down2) [right] {$-3$};
 	\draw (0,-2*\vedge )  node [left] {$E_1$};
 	\draw (\edge,0)  node (e+1-up) [above] {$-2$};
 	\draw (\edge,0)  node [below] {$E_5$};
 	\draw (\edge + \dashedge,0)  node (e-last-1-up) [above] {$-2$};
 	\draw (\edge + \dashedge,0)  node [below] {$E_{10}$};
 	\draw (2*\edge + \dashedge,0)  node (e-last-up) [above] {$-2$};
 	\draw (2*\edge + \dashedge,0)  node [below] {$E_{11}$};
 	
 	\pgfmathsetmacro\factorr{.2}
 	\draw [thick, decoration={brace, mirror, raise=15pt},decorate] (\edge + \edge*\factorr,0) -- (\edge + \dashedge - \edge*\factorr,0);
 	\draw (\edge + 0.5*\dashedge,-\edge) node [text width= 2.5cm, align = center] (extranodes) {string of vertices of weight $-2$};
 			 	
\end{tikzpicture}
}
\subcaptionbox[]{%
    Weighted dual graph of the minimal resolution of the singularity of $Y'_i$%
    \label{fig:non-example-b}%
}
[%
    0.48\textwidth 
]%
{%
\begin{tikzpicture}[scale=1, font = \small] 

	\pgfmathsetmacro\factor{1.25}
	\pgfmathsetmacro\dashedge{2*\factor}	
	\pgfmathsetmacro\edge{.75*\factor}
	\pgfmathsetmacro\vedge{.55*\factor}

	\draw[thick] (-\edge,0) -- (\edge,0);
	\draw[thick] (0,0) -- (0,-2*\vedge);
	\draw[thick, dashed] (\edge,0) -- (\edge + \dashedge,0);
	\draw[thick] (\edge + \dashedge,0) -- (2*\edge + \dashedge,0);
	
	\fill[black] (-\edge, 0) circle (3pt);
	\fill[black] (0, 0) circle (3pt);
	\fill[black] (0, -\vedge) circle (3pt);
	\fill[black] (0, - 2*\vedge) circle (3pt);
	\fill[black] (\edge, 0) circle (3pt);
	\fill[black] (\edge + \dashedge, 0) circle (3pt);
	\fill[black] (2*\edge + \dashedge, 0) circle (3pt);
	
	\draw (-\edge,0 )  node (e2up) [above] {$-2$};
	\draw (-\edge,0 )  node [below] {$E_2$};
	\draw (0,0 )  node (e4up) [above] {$-2$};	
	\draw (0,0 )  node [below left] {$E_4$};
	\draw (0,-\vedge )  node (down1) [right] {$-2$};
	\draw (0,-\vedge )  node [left] {$E_3$};
	\draw (0, -2*\vedge)  node (down2) [right] {$-3$};
	\draw (0,-2*\vedge )  node [left] {$E_1$};
	\draw (\edge,0)  node (e+1-up) [above] {$-2$};
	\draw (\edge,0)  node [below] {$E_5$};
	\draw (\edge + \dashedge,0)  node (e-last-1-up) [above] {$-2$};
	\draw (\edge + \dashedge,0)  node [below] {$E_{10}$};
	\draw (2*\edge + \dashedge,0)  node (e-last-up) [above] {$-2$};
	\draw (2*\edge + \dashedge,0)  node [below] {$E_{11}$};
	
	\pgfmathsetmacro\factorr{.2}
	\draw [thick, decoration={brace, mirror, raise=15pt},decorate] (\edge + \edge*\factorr,0) -- (\edge + \dashedge - \edge*\factorr,0);
	\draw (\edge + 0.5*\dashedge,-\edge) node [text width= 2.5cm, align = center] (extranodes) {string of vertices of weight $-2$}; 			 	
\end{tikzpicture}
}%
\caption{Singularity of $Y'_i$}
\label{fig:non-example}
\end{figure}
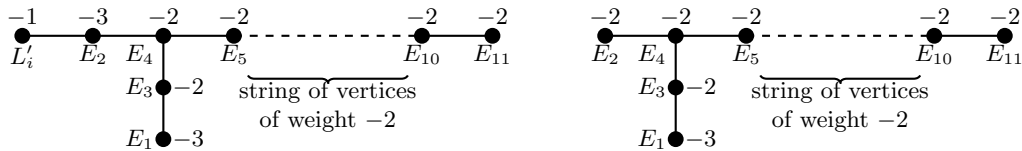

Note that the pairs of germs $(C_1,L)$ and $(C_2, L)$ are {\em isomorphic} via the map $(u,v) \mapsto (u,v+u^2)$. It follows that `weighted dual graphs' (Definition \ref{dual-defn}) of $E'^{(i)}$'s are {\em identical}; they are depicted in Figure \ref{fig:non-example-a} (we labeled the vertices according to the order of appearance of the corresponding curves in the sequence of blow-ups). It is straightforward to compute that the matrices of intersection numbers of the components of $E'^{(i)}$'s are negative definite, so that there is a bimeromorphic analytic map $Y_i \to Y'_i$ contracting $E'^{(i)}$. Note that each $Y'_i$ is a normal analytic surface with one singular point $P_i$. It follows from the construction that the weighted dual graphs of the minimal resolution of singularities of $Y'_i$ are identical (see Figure \ref{fig:non-example-b}), so that the numerical invariants of the singularities of $Y'_i$'s are also {\em identical}.\\

\begin{figure}[htp]
\centering

\includegraphics[height=3cm]{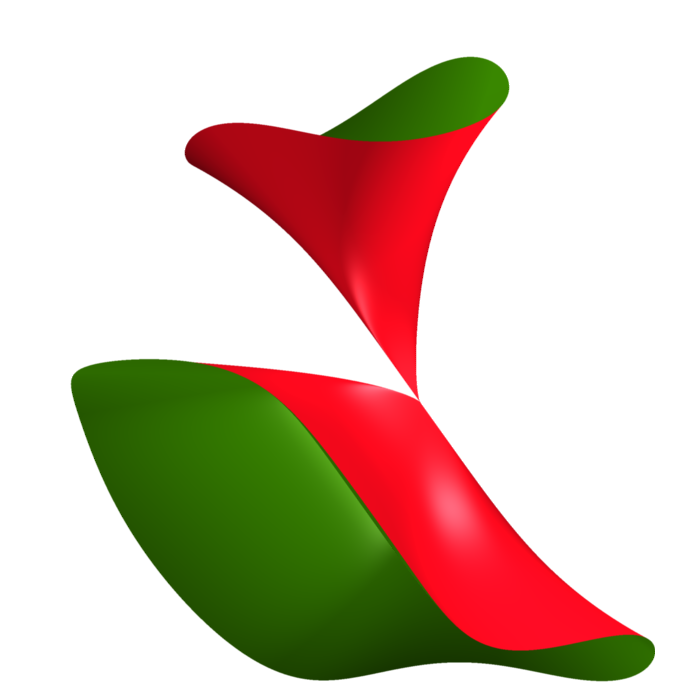}

\caption[]{The singularity of $z^2 = x^5 + xy^5$ at the origin ({\em  whirling dervish})}
\label{fig:dervish}
\end{figure}

In fact it follows (from e.g.\ \cite[Section 8]{neumann}) that the singularities of $Y'_i$'s are also {\em homeomorphic}. However, Theorem \ref{effective-answer} and Example \ref{non-key-example} imply that $Y'_1$ is algebraic, but $Y'_2$ is {\em not}.


\begin{rem} \label{simplest-remark}
It is straightforward to verify that the weighted dual graph of Figure \ref{fig:non-example-b} is precisely the graph labeled $D_{9,*,0}$ in \cite{laufer-elliptic}.
It then follows from \cite{laufer-elliptic} that the singularities at $P_i$ are {\em Goerenstein hypersurface singularities} of multiplicity $2$ and geometric genus $1$, which are also {\em minimally elliptic} (in the sense of \cite{laufer-elliptic}). Minimally elliptic Gorenstein singularities have been extensively studied (see e.g.\ \cite{yau}, \cite{ohyanagi}, \cite{nemethi-weakly-elliptic}), and in a sense they form the simplest class of non-rational singularities\footnote{Indeed, every connected proper subvariety of the exceptional divisor of the minimal resolution of a minimally elliptic singularity is the exceptional divisor of the minimal resolution of a rational singularity \cite{laufer-elliptic}.}. Since having only rational singularities imply algebraicity of the surface \cite{artractability}, it follows that the surface $Y'_2$ we constructed above is a normal non-algebraic Moishezon surface with the `simplest possible' singularity. 
\end{rem} 

It follows from \cite[Table 2]{laufer-elliptic} that the singularity at the origin of $z^2 = x^5 + xy^5$ (Figure \ref{fig:dervish}) is of the same type as the singularity of each $Y'_i$, $1 \leq i \leq 2$.

\section{Background I} \label{firstground}
Here we compile the background material needed to state the results. In Section \ref{secondground} we compile further background material that we use for the proof.

\begin{notation} \label{notation-x}
	Throughout the rest of the article we use $X$ to denote $\cc^2$ with coordinate ring $\cc[x,y]$ and $\Xxy$ to denote copy of $\pp^2$ such that $X$ is embedded into $\Xxy$ via the map $(x,y) \mapsto [1:x:y]$. We also denote by $L_\infty$ the {\em line at infinity} $\Xxy\setminus X$, and by $Q_y$ the point of intersection of $L_\infty$ and (closure of) the $y$-axis. Finally, if $\omega_0, \ldots, \omega_n$ are positive integers, we denote by $\pp^n(\omega_0, \ldots, \omega_n)$ the complex $n$-dimensional weighted projective space corresponding to weights $\omega_0, \ldots, \omega_n$. 
\end{notation}

\subsection{Meromorphic and descending Puiseux series}

\begin{defn}[Meromorphic Puiseux series] \label{meromorphic-defn}
	A {\em meromorphic Puiseux series} in a variable $u$ is a fractional power series of the form $\sum_{m \geq M} a_m u^{m/p}$ for some $m,M \in \zz$, $p  \geq 1$ and $a_m \in \cc$ for all $m \in \zz$. If all exponents of $u$ appearing in a meromorphic Puiseux series are positive, then it is simply called a {\em Puiseux series} (in $u$). Given a meromorphic Puiseux series $\phi(u)$ in $u$, write it in the following form:
	\begin{align*} 
	\phi(u) =  \cdots+ a_1 u^{\frac{q_1}{p_1}} + \cdots +  a_2 u^{\frac{q_2}{p_1p_2}} + \cdots 
	+ a_l u^{\frac{q_l}{p_1p_2 \cdots p_l}} + \cdots 
	\end{align*}
	where $q_1/p_1$ is the smallest non-integer exponent, and for each $k$, $1 \leq k \leq l$, we have that $a_k \neq 0$, $p_k \geq 2$, $\gcd(p_k, q_k) = 1$, and the exponents of all terms with order between $\frac{q_k}{p_1\cdots p_k}$ and $\frac{q_k}{p_1\cdots p_{k+1}}$ (or, if $k = l$, then all terms of order $> \frac{1}{p_1\cdots p_l}$) belong to $\frac{1}{p_1 \cdots p_{k}}\zz$. Then the pairs $(q_1, p_1), \ldots, (q_l, p_l)$, are called the {\em Puiseux pairs} of $\phi$ and the exponents $\frac{q_k}{p_1 \cdots p_k}$, $1 \leq k \leq l$, are called {\em characteristic exponents} of $\phi$. The {\em polydromy order} \cite[Chapter 1]{casas-alvero} of $\phi$ is $p := p_1\cdots p_l$, i.e.\ the polydromy order of $\phi$ is the smallest $p$ such that $\phi \in \cc((u^{1/p}))$. Let $\zeta$ be a primitive $p$-th root of unity. Then the {\em conjugates} of $\phi$ are
	\begin{align*}
	\phi_j(u) :=  \cdots+ a_1 \zeta^{jq_1p_2\cdots p_l} u^{\frac{q_1}{p_1}} + \cdots +  a_2 \zeta^{jq_2p_3\cdots p_l} u^{\frac{q_2}{p_1p_2}} + \cdots 
	+ a_l \zeta^{jq_l} u^{\frac{q_l}{p_1p_2 \cdots p_l}} + \cdots
	\end{align*}
	for $1 \leq j \leq p$ (i.e.\ $\phi_j$ is constructed by multiplying the coefficients of terms of $\phi$ with order $n/p$ by $\zeta^{jn}$).
\end{defn}

We recall the standard fact that the field of meromorphic Puiseux series in $u$ is the algebraic closure of the field $\cc((u))$ of Laurent polynomials in $u$:

\begin{thm} \label{puiseux-thm}
	Let $f \in \cc((u))[v]$ be an irreducible monic polynomial in $v$ of degree $d$. Then there exists a meromorphic Puiseux series $\phi(u)$ in $u$ of polydromy order $d$ such that 
	\begin{align*}
	f = \prod_{i=1}^d (v - \phi_i(u)),
	\end{align*}
	where $\phi_i$'s are conjugates of $\phi$. 
\end{thm}

\begin{defn}[Descending Puiseux Series] \label{dpuiseux}
	A {\em descending Puiseux series} in $x$ is a meromorphic Puiseux series in $x^{-1}$. The notions regarding meromorphic Puiseux series defined in Definition \ref{meromorphic-defn} extend naturally to the setting of descending Puiseux series. In particular, if $\phi(x)$ is a descending Puiseux series and the Puiseux pairs of $\phi(1/x)$ are $(q_1, p_1), \ldots, (q_l, p_l)$, then $\phi$ has Puiseux pairs $(-q_1, p_1), \ldots, (-q_l, p_l)$, polydromy order $p := p_1 \cdots p_l$, and characteristic exponents $-q_k/(p_1 \cdots p_k)$ for $1 \leq k \leq l$.   
\end{defn}

\begin{notation} \label{descending-notation}
We use $\dpsxc$ to denote the field of descending Puiseux series in $x$. For $\phi \in \dpsxc$ and $r \in \rr$, we denote by $[\phi]_{> r}$ the {\em descending Puiseux polynomial} (i.e.\ descending Puiseux series with finitely many terms) consisting of all terms of $\phi$ of degree $> r$. If $\psi$ is also in $\dpsxc$, then we write $\phi \equiv_r \psi$ iff $[\phi]_{>r} = [\psi]_{> r}$ iff $\deg_x(\phi- \psi) \leq r$. 
\end{notation}

The following is an immediate Corollary of Theorem \ref{puiseux-thm}:

\begin{thm} \label{dpuiseux-factorization}
	Let $f \in \cc[x,x^{-1}, y]$. Then there are (up to conjugacy) unique descending Puiseux series $\phi_1, \ldots, \phi_k$ in $x$, a unique non-negative integer $m$ and $c \in \cc^*$ such that
	$$f = cx^m \prod_{i=1}^k \prod_{\parbox{1.75cm}{\scriptsize{$\phi_{ij}$ is a con\-ju\-ga\-te of $\phi_i$}}}\mkern-27mu \left(y - \phi_{ij}(x)\right)$$
\end{thm}

\subsection{Divisorial discrete valuation and semidegree} \label{divisorial-section}
Let $\sigma:\tilde Y \dashrightarrow Y$ be a birational correspondence of normal complex algebraic surfaces and $C$ be an irreducible analytic curve on $\tilde Y$. Then the local ring $\sheaf_{\tilde Y ,C}$ of $C$ on $\tilde Y$ is a discrete valuation ring. Let $\nu$ be the associated valuation on the field $\cc(Y)$ of rational functions on $Y$; in other words $\nu$ is the order of vanishing along $C$. We say that $\nu$ is a {\em divisorial discrete valuation} on $\cc(Y)$; the {\em center} of $\nu$ on $Y$ is $\sigma(C\setminus S)$, where $S$ is the set of points of indeterminacy of $\sigma$ (the normality of $Y$ ensures that $S$ is a discrete set, so that $C \setminus S \neq \emptyset$). Moreover, if $U$ is an open subset of $Y$, we say that $\nu$ is {\em centered at infinity} with respect to $U$ iff $\sigma(C\setminus S) \subseteq Y\setminus U$. 

\begin{defn}[Semidegree] \label{semi-defn}
Let $U$ be an affine variety and $\nu$ be a divisorial discrete valuation on the ring $\cc[U]$ of regular functions on $U$ which is centered at infinity with respect to $U$. Then we say that $\delta := -\nu$ is a {\em semidegree} on $\cc[U]$.
\end{defn}

The following result, which connects semidegrees on $\cc[x,y]$ with descending Puiseux series in $x$ is a reformulation of \cite[Proposition 4.1]{favsson-tree}.

\begin{thm} \label{semidescending}
Let $\delta$ be a semidegree on $\cc[x,y]$. Assume that $\delta(x) > 0$. Then there is a {\em descending Puiseux polynomial} (i.e.\ a descending Puiseux series with finitely many terms) $\phi_\delta(x)$ (unique up to conjugacy) in $x$ and a (unique) rational number $r_\delta <  \ord_x(\phi_\delta)$ such that for every $f \in \cc[x,y]$, 
\begin{align}
\delta(f) = \delta(x)\deg_x(f(x, \phi_\delta(x) + \xi x^{r_\delta})), \label{phi-delta-defn}
\end{align}
where $\xi$ is an indeterminate. 
\end{thm}

\begin{defn} \label{generic-descending-defn}
	In the situation of Theorem \ref{semidescending}, we say that $\tilde \phi_\delta(x,\xi):= \phi_\delta(x) + \xi x^{r_\delta}$ is the {\em generic descending Puiseux series} associated to $\delta$. Moreover, if $\bar X$ is an analytic compactification of $X = \cc^2$ and $Z \subseteq \bar X \setminus \cc^2$ is a curve at infinity such that $\delta$ is the order of pole along $Z$, then we also say that $\tilde \phi_\delta(x,\xi)$ is the {\em generic descending Puiseux series associated to $Z$}.
\end{defn}

\begin{example} \label{wt-depuiseux}
	If $\delta$ is a weighted degree in $(x,y)$-coordinates corresponding to weights $p$ for $x$ and $q$ for $y$ with $p,q$ positive integers, then the generic descending Puiseux series corresponding to $\delta$ is $\tilde \phi_\delta = \xi x^{q/p}$. Note that if we embed $\cc^2 = \spec\cc[x,y]$ into the weighted projective space $\pp^2(1,p,q)$ via $(x,y) \mapsto [1:x:y]$, then $\delta$ is precisely the order of the pole along the curve at infinity.
\end{example}

\begin{example} \label{non-descending-example}
Recall the set up of the example from Section \ref{non-example-section}. Then $C_i$'s have Puiseux expansions $v = \psi_i(u)$ at $O$, where 
\begin{align*}
\psi_1(u) &= u^{3/5},\\
\psi_2(u) &= u^{3/5} + u^2.
\end{align*}
Now note that $(x,y) := (1/u,v/u)$ are coordinates on $\pp^2\setminus L \cong \cc^2$, and with respect to $(x,y)$ coordinates the $C_1$ has a {\em descending} Puiseux expansion of the form $y = x\psi_1(1/x) = x^{2/5}$. Similarly, $C_2$ has a descending Puiseux expansion of the form $y = x\psi_2(1/x) = x^{2/5} + x^{-1}$. Let $\delta_i$ be the order of pole along $E^*_i$, $1 \leq i \leq 2$. Then the generic descending Puiseux series corresponding to $\delta_1$ and $\delta_2$ are respectively of the form
\begin{align}
\begin{split}
\tilde \phi_{\delta_1}(x,\xi_1) &= x^{2/5} + \xi_1 x^{-6/5}, \\
\tilde \phi_{\delta_2}(x,\xi_2) &= x^{2/5} + x^{-1} + \xi_2 x^{-6/5}. 
\end{split} \label{non-phi-delta}
\end{align} 
\end{example}

\begin{defn}[Formal Puiseux pairs of generic descending Puiseux series] \label{formal-defn}
	Let $\delta$ and $\tilde \phi_\delta(x,\xi):= \phi_\delta(x) + \xi x^{r_\delta}$ be as in Definition \ref{generic-descending-defn}. Let the Puiseux pairs of $\phi_\delta$ be $(q_1, p_1), \ldots, (q_l,p_l)$. Express $r_\delta$ as $q_{l+1}/(p_1 \cdots p_lp_{l+1})$ where $p_{l+1} \geq 1$ and $\gcd(q_{l+1}, p_{l+1}) = 1$. Then the {\em formal Puiseux pairs} of $\tilde \phi_\delta$ are $(q_1, p_1), \ldots, (q_{l+1},p_{l+1})$, with $(q_{l+1}, p_{l+1})$ being the {\em generic} formal Puiseux pair. Note that 
	\begin{enumerate}
		\item $\delta(x) = p_1 \cdots p_{l+1}$,
		\item it is possible that $p_{l+1} = 1$ (as opposed to other $p_k$'s, which are always $\geq 2$). 
	\end{enumerate}
\end{defn}

\subsection{Geometric interpretation of generic descending Puiseux series}
In this subsection we recall from \cite{sub2-1} the geometric interpretation of generic descending Puiseux series. We use the notations introduced in Notation \ref{notation-x}.
\begin{defn}
An {\em irreducible analytic curve germ at infinity} on $X$ is the image $\gamma$ of an analytic map $h$ from a punctured neighborhood $\Delta'$ of the origin in $\cc$ to $X$ such that $|h(s)| \to \infty$ as $|s| \to 0$ (in other words, $h$ is analytic on $\Delta'$ and has a pole at the origin). If $\bar X$ is an analytic compactification of $X$, then there is a unique point $P \in \bar X \setminus X$ such that $|h(s)| \to P$ as $|s| \to 0$. We call $P$ the {\em center} of $\gamma$ on $\bar X$, and write $P = \lim_{\bar X} \gamma$. 
\end{defn}

Let $\bar X$ be a primitive normal analytic compactification of $X$ with an irreducible curve $C_\infty$ at infinity. Let $\sigma: \Xxy\dashrightarrow \bar X$ be the natural bimeromorphic map, and let $Y$ be a {\em resolution of indeterminacies} of $\sigma$, i.e.\ $Y$ is a non-singular rational surface equipped with analytic maps $\pi: Y \to \Xxy$ and $\pi':Y \to \bar X$ such that $\pi' = \sigma \circ \pi$. Let $L'_\infty$ be the strict transform of $L_\infty \subseteq \Xxy$ on $Y$ and $Q'_y \in L'_\infty$ be (the unique point) such that $\pi(Q'_y) = Q_y$. Let $P_\infty := \pi'(Q'_y) \in C_\infty$. 

\begin{prop}[{\cite[Proposition 3.5]{sub2-1}}] \label{prop-param}
Let $\delta$ be the order of pole along $C_\infty$, $\tilde \phi_\delta(x,\xi)$ be the generic descending Puiseux series associated to $\delta$ and $\gamma$ be an irreducible analytic curve germ on $X$. Then $ \lim_{\bar X} \gamma \in C_\infty\setminus\{P_\infty\}$ iff $\gamma$ has a parametrization of the form 
\begin{align}
t \mapsto (t, \tilde \phi_\delta(t,\xi)|_{\xi = c} + \ldt)\quad \text{for}\ |t| \gg 0 \label{dpuiseux-param} \tag{$*$}
\end{align}
for some $c \in \cc$, where $\ldt$ means `lower degree terms' (in $t$).
\end{prop}

\begin{reminition} \label{P^2-center}
We call $P_\infty$ {\em a center of $\pp^2$-infinity on $\bar X$.} $P_\infty$ is in fact {\em unique} in the case of `generic' primitive normal compactifications of $\cc^2$ (we do not use this uniqueness in this article, so we state it without a proof):
\begin{itemize}
\item If $\bar X \cong \pp^2(1,1,q)$ for some $q > 0$, then {\em every} point of $C_\infty$ is a center of $\pp^2$-infinity on $\bar X$.
\item If $\bar X \cong \pp^2(1,p,q)$ for some $p,q > 1$, then $\bar X$ has {\em two} singular points, and these are precisely the centers of $\pp^2$-infinity on $\bar X$.
\item In all other cases, there is a {\em unique} center of $\pp^2$-infinity on $\bar X$ - it is precisely the unique point on $\bar X$ which has a non-quotient singularity.
\end{itemize}
\end{reminition}

\subsection{Key forms of a semidegree}
Let $\delta$ be a semidegree on $\cc[x,y]$ such that $\delta(x) > 0$. Pick $k > 1$ such that $\delta(y/x^k) < 0$. Set $(u,v) := (1/x,y/x^k)$. Then $\nu := -\delta$ is a discrete valuation on $\cc[u,v]$ which is {\em centered at the origin}. It follows that $\nu$ can be completely described in terms of a finite sequence of {\em key polynomials} in $(u,v)$ \cite{maclane-key}. The {\em key forms} of $\delta$ that we introduce in this section are precisely the analogue of key polynomials of $\nu$. We refer to \cite[Chapter 2]{favsson-tree} for the properties of key polynomials that we used as a model for our definition of key forms:


\begin{defn}[Key forms] \label{key-defn}
Let $\delta$ be a semidegree on $\cc[x,y]$ such that $\delta(x) > 0$. A sequence of elements $g_0, g_1, \ldots, g_{n+1} \in \cc[x,x^{-1},y]$ is called the sequence of {\em key forms} for $\delta$ if the following properties are satisfied with $\eta_j := \delta(g_j)$, $0 \leq j \leq n+1$: 
\begin{compactenum}
\let\oldenumi\theenumi
\renewcommand{\theenumi}{P\oldenumi}
\addtocounter{enumi}{-1}
\item \label{semigroup-property} $\eta_{j+1} < \alpha_j \eta_j = \sum_{i = 0}^{j-1}\beta_{j,i}\eta_i$ for $1 \leq j \leq n$, where 
\begin{compactenum}
\item $\alpha_j = \min\{\alpha \in \zz_{> 0}: \alpha\eta_j \in \zz \eta_0 + \cdots + \zz \eta_{j-1}\}$ for $1 \leq j \leq n$,
\item $\beta_{j,i}$'s are integers such that $0 \leq \beta_{j,i} < \alpha_i$ for $1 \leq i < j \leq n$ (in particular, only $\beta_{j,0}$'s are allowed to be negative). 
\end{compactenum}
\item $g_0 = x$, $g_1 = y$.
\item \label{next-property} For $1 \leq j \leq n$, there exists $\theta_j \in \cc^*$ such that 
\begin{align*}
g_{j+1} = g_j^{\alpha_j} - \theta_j g_0^{\beta_{j,0}} \cdots g_{j-1}^{\beta_{j,j-1}}.
\end{align*}

\item \label{generating-property} Let $z_1, \ldots, z_{n+1}$ be indeterminates and $\eta$ be the {\em weighted degree} on $B := \cc[x,x^{-1},z_1, \ldots, z_{n+1}]$ corresponding to weights  $\eta_0$ for $x$ and $\eta_j$ for $z_j$, $1 \leq j \leq n+1$ (i.e.\ the value of $\eta$ on a polynomial is the maximum `weight' of its monomials). Then for every polynomial $g \in \cc[x,x^{-1},y]$, 
\begin{align}
\delta(g) = \min\{\eta(G): G(x,z_1, \ldots, z_{n+1}) \in B,\ G(x,g_1, \ldots, g_{n+1}) = g\}. \label{generating-eqn}
\end{align}
\end{compactenum} 
\end{defn}
The properties of key forms of semidegrees compiled in the following theorem are straightforward analogues of corresponding (standard) properties of key polynomials of valuations.

\begin{thm} \label{key-thm}
\mbox{}
\begin{enumerate}
\item Every semidegree $\delta$ on $\cc[x,y]$ such that $\delta(x) > 0$ has a unique and finite sequence of key forms.
\item \label{key-to-degree} Conversely, given $g_0, \ldots, g_{n+1} \in \cc[x,x^{-1},y]$ and integers $\eta_0, \ldots, \eta_{l+1}$ with $\eta_0 > 0$ which satisfy properties \eqref{semigroup-property}--\eqref{next-property}, there is a unique semidegree $\delta$ on $\cc[x,y]$ such that $g_j$'s are key forms of $\delta$ and $\eta_j = \delta(g_j)$, $0 \leq j \leq n+1$.
\item Recall Notation \ref{notation-x}. Assume $\sigma: \bar X^* \to \Xxy$ is a composition of point blow-ups and $E^* \subseteq \bar X^*$ is an exceptional curve of $\sigma$. Let $\delta$ be the order of pole along $E^*$. Assume $\delta(x) > 0$. Then the following data are equivalent (i.e.\ given any one of them, there is an explicit algorithm to construct the others in finite time):
\begin{enumerate}
\item a minimal sequence of points on successive blow-ups of $\Xxy$ such that $\sigma$ factors through the composition of these blow-ups and $E^*$ is the strict transform of the exceptional curve of the last blow-up.
\item a generic descending Puiseux series of $\delta$.
\item the sequence of key forms of $\delta$. 
\end{enumerate}
\item \label{last-expansion} Let $\delta$ be a semidegree on $\cc[x,y]$ such that $\delta(x) > 0$. Let $\tilde \phi_\delta(x,\xi):= \phi_\delta(x) + \xi x^{r_\delta}$ be the generic descending Puiseux series and $g_{n+1}$ be the last key form of $\delta$. Then the descending Puiseux factorization of $g_{n+1}$ is of the form
\begin{align*}
g_{n+1} = \prod_{\parbox{1.5cm}{\scriptsize{$\psi_{j}$ is a con\-ju\-ga\-te of $\psi$}}}\mkern-18mu \left(y - \psi_{j}(x)\right)
\end{align*}
for some $\psi \in \dpsxc$ such that $\psi \equiv_{r_\delta} \phi_\delta$ (see Notation \ref{descending-notation}).
\end{enumerate}
\end{thm}

\begin{example} \label{wt-key-forms}
Let $\delta$ be the weighted degree from Example \ref{wt-depuiseux}. The key forms of $\delta$ are $g_0 = x$ and $g_1 = y$. 
\end{example}

\begin{example} \label{non-key-example}
Let $\delta_1$ and $\delta_2$ be the semidegrees from Example \ref{non-descending-example}. Then the key forms of $\delta_1$ are $x,y, y^5 - x^2$. On the other hand the key forms of $\delta_2$ are $x, y, y^5 - x^2, y^5 - x^2 - 5x^{-1}y^4$ (see Algorithm \ref{key-algorithm} for the general algorithm to compute key forms from generic descending Puiseux series).
\end{example}


\begin{defn}[Essential key forms] \label{essential-defn}
Let $\delta$ be a semidegree on $\cc[x,y]$ such that $\delta(x) > 0$, and let $g_0, \ldots, g_{n+1}$ be the key forms of $\delta$. Pick the subsequence $j_1, j_2, \ldots, j_m$ of $1, \ldots, n$ consisting of all $j_k$'s such that $\alpha_{j_k} > 1$ (where $\alpha_{j_k}$ is as in Property \ref{semigroup-property} of Definition \ref{key-defn}). Set 
\begin{align*}
f_k &:= \begin{cases}
g_0 = x & \text{if}\ k = 0,\\
g_{j_k} & \text{if}\ 1 \leq k \leq m,\\
g_{n+1}	& \text{if}\ k = m+1.
\end{cases}
\end{align*}
We say that $f_0, \ldots, f_{m+1}$ are the {\em essential key forms} of $\delta$. The following properties of essential key forms follow in a straightforward manner from the defining properties of key forms:
\end{defn}

\begin{prop} \label{essential-value-prop}
Let the notations be as in Definition \ref{essential-defn}. Let $\tilde \phi_\delta(x,\xi)$ be the generic descending Puiseux series of $\delta$ and $(q_1, p_1), \ldots, (q_{l+1}, p_{l+1})$ be the formal Puiseux paris of $\tilde \phi_\delta$. Then 
\begin{enumerate}
\item $l = m$, i.e.\ the number of essential key forms of $\delta$ is precisely $l+1$.
\item \label{omega-from-p} Set $\omega_k := \delta(f_k)$, $0 \leq k \leq l+1$. Then the sequence $\omega_0, \ldots, \omega_{l+1}$ depend only on the formal Puiseux pairs of $\tilde \phi_\delta$. More precisely,  with $p_0 := q_0 := 1$, we have 
\begin{align}
\omega_k
	&= 
	\begin{cases}
	p_1 \cdots p_{l+1} & \text{if}\ k = 0,\\
	p_{k-1}\omega_{k-1} + (q_k - q_{k-1}p_k)p_{k+1} \cdots p_{l+1} & \text{if}\ 1 \leq k \leq l+1.
	\end{cases}
	\label{omega-defn}
\end{align}
\item \label{alpha=p} Let $\alpha_1, \ldots, \alpha_{n+1}$ be as in Property \eqref{semigroup-property} of key forms. Then 
\begin{align*}
\alpha_j &= \begin{cases}
p_k & \text{if}\ j = j_k,\ 1 \leq  k \leq l+1, \\
1	& \text{otherwise.}
\end{cases}
\end{align*}
\item \label{essential-group} Pick $j$, $0 \leq j \leq n+1$. Assume $j_k < j < j_{k+1}$ for some $k$, $0 \leq k \leq l$. Then $\delta(g_j)$ is in the group generated by $\omega_0, \ldots, \omega_k$.
\end{enumerate}
\end{prop}
\begin{defn}
We call $\omega_0, \ldots, \omega_{l+1}$ of Proposition \ref{essential-value-prop} the sequence of {\em essential key values} of $\delta$.
\end{defn}

\begin{example} \label{non-essential-example}
Let $\delta_1, \delta_2$ be as in Examples \ref{non-descending-example} and \ref{non-key-example}. Then all the key forms of $\delta_1$ are essential, and essential key values are $\omega_0 = \delta_1(x) = 5$, $\omega_1 = \delta_1(y) = 2$, $\omega_2 = \delta_1(y^5-x^2) = 2$. The key forms of $\delta_2$ are $x,y, y^5 - x^2 - 5x^{-1}y^4$. The sequence of essential key values of $\delta_2$ is the {\em same} as that of $\delta_1$.
\end{example}


\subsection{Resolution of singularities of primitive normal compactifications}

Given two birational algebraic surfaces $Y_1,Y_2$, we say that $Y_1$ {\em dominates} $Y_2$ if the birational map $Y_1 \dashrightarrow Y_2$ is in fact a morphism. Let $\bar X$ be a primitive normal analytic compactification of $X := \cc^2$ and $\pi:Y \to \bar X$ be a resolution of singularities of $\bar X$. We say that $\pi$ or $Y$ is {\em $\pp^2$-dominating} if $Y$ dominates $\pp^2$.
$\pi$ is a {\em minimal} $\pp^2$-dominating resolution of singularities of $\bar X$ if up to isomorpshism (of algebraic varieties) $Y$ is the only $\pp^2$-dominating resolution of singularities of $\bar X$ which is dominated by $Y$.

\begin{thm} \label{unique-resolution}
Every primitive normal analytic compactification of $\cc^2$ has a unique minimal $\pp^2$-dominating resolution of singularities. 
\end{thm}

We have not found any proof of Theorem \ref{unique-resolution} in the literature. We give a proof in \cite{sub2-2} (using Theorem \ref{effective-answer} of this article). In this section we recall from \cite{sub2-1} a description of the {\em dual graphs} of minimal $\pp^2$-dominating resolutions of singularities of primitive normal analytic compactifications of $\cc^2$. 

\begin{defn} \label{dual-defn}
	Let $E_1, \ldots, E_k$ be non-singular curves on a (non-singular) surface such that for each $i \neq j$, either $E_i \cap E_j = \emptyset$, or $E_i$ and $E_j$ intersect transversally at a single point. Then $E = E_1 \cup \cdots \cup E_k$ is called a {\em simple normal crossing curve}. The {\em (weighted) dual graph} of $E$ is a weighted graph with $k$ vertices $V_1, \ldots, V_k$ such that 
	\begin{itemize}
		\item there is an edge between $V_i$ and $V_j$ iff $E_i \cap E_j \neq \emptyset$,
		\item the weight of $V_i$ is the self intersection number of $E_i$.
	\end{itemize}
	Usually we will abuse the notation, and label $V_i$'s also by $E_i$. 
\end{defn}

\begin{defn} \label{augmented-defn}
	Let $\bar X$ be a primitive normal analytic compactification of $X := \cc^2$ and $\pi: Y \to \bar X$ be a resolution of singularities of $\bar X$ such that $Y \setminus X$ is a simple normal crossing curve. The {\em augmented dual graph} of $\pi$ is the dual graph (Definition \ref{dual-defn}) of $Y\setminus X$. If $Y$ is $\pp^2$-dominating, we define the {\em augmented and marked dual graph} of $\pi$ to be its augmented dual graph with the strict transforms of the curves at infinity on $\pp^2$ and $\bar X$ marked (e.g.\ by different colors or labels). 
\end{defn}

Given a sequence $(\tilde q_1, \tilde p_1), \ldots, (\tilde q_n, \tilde p_n)$ of pairs of relatively prime integers, and positive integers $m,e$ such that $1 \leq m \leq n$, we denote by $\tilde \Gamma_{\vec{\tilde q}, \vec {\tilde p}, m, e}$ the weighted graph in figure \ref{curve-resolution},
\begin{figure}[htp]
	\newcommand{\curveresolutionblock}[6]{
		\pgfmathsetmacro\x{0}
		\pgfmathsetmacro\y{0}
		
		\draw[thick] (\x - \edge,\y) -- (\x,\y);
		\draw[thick] (\x,\y) -- (\x+\edge,\y);
		\draw[thick] (\x,\y) -- (\x, \y-\vedge);
		\draw[thick, dashed] (\x, \y-\vedge) -- (\x, \y-\vedge - \dashedvedge);
		\draw[thick] (\x, \y-\vedge - \dashedvedge) -- (\x, \y-2*\vedge - \dashedvedge);
		
		\fill[black] (\x - \edge, \y) circle (3pt);
		\fill[black] (\x, \y) circle (3pt);
		\fill[black] (\x + \edge, \y) circle (3pt);
		\fill[black] (\x, \y- \vedge) circle (3pt);
		\fill[black] (\x, \y- \vedge - \dashedvedge) circle (3pt);
		\fill[black] (\x, \y- 2*\vedge - \dashedvedge) circle (3pt);
		
		\draw (\x- \edge,\y )  node (left) [above] {$-{#1}_{#5}^{{#2}_{#5}}$};
		\draw (\x,\y )  node (center) [above] {$-{#1}_{#6}^0 - 1$};
		\draw (\x+ \edge,\y )  node (right) [above] {$-{#1}_{#6}^1$};
		\draw (\x,\y- \vedge)  node (bottom1) [right] {$-{#3}_{#5}^{{#4}_{#5}}$};
		\draw (\x,\y- \vedge - \dashedvedge)  node (bottom2) [right] {$-{#3}_{#5}^2$};
		\draw (\x,\y- 2*\vedge - \dashedvedge)  node (bottom3) [right] {$-{#3}_{#5}^1$};
	}
	\begin{center}
		\begin{tikzpicture}
		\pgfmathsetmacro\edge{1.5}
		\pgfmathsetmacro\vedge{.75}
		\pgfmathsetmacro\dashedvedge{1}
		
		\draw[thick, dashed] (0,0) -- (\edge,0);
		\fill[black] (0, 0) circle (3pt);
		\draw (0,0)  node (first) [above] {$-u_1^1$};
		
		\begin{scope}[shift={(2*\edge,0)}]
		\curveresolutionblock{u}{t}{v}{r}{1}{2}
		\end{scope}
		
		\draw[thick, dashed] (3*\edge,0) -- (4*\edge,0);
		
		\begin{scope}[shift={(5*\edge,0)}]
		\curveresolutionblock{u}{t}{v}{r}{m-1}{m}
		\end{scope}	
		
		\draw[thick, dashed] (6*\edge,0) -- (7*\edge,0);	
		
		\pgfmathsetmacro\x{8*\edge}
		\pgfmathsetmacro\y{0}
		
		\draw[thick] (\x - \edge,\y) -- (\x,\y);
		\draw[thick] (\x,\y) -- (\x, \y-\vedge);
		\draw[thick, dashed] (\x, \y-\vedge) -- (\x, \y-\vedge - \dashedvedge);
		\draw[thick] (\x, \y-\vedge - \dashedvedge) -- (\x, \y-2*\vedge - \dashedvedge);
		
		\fill[black] (\x - \edge, \y) circle (3pt);
		\fill[black] (\x, \y) circle (3pt);
		\fill[black] (\x, \y- \vedge) circle (3pt);
		\fill[black] (\x, \y- \vedge - \dashedvedge) circle (3pt);
		\fill[black] (\x, \y- 2*\vedge - \dashedvedge) circle (3pt);
		
		\draw (\x- \edge,\y )  node (left) [above] {$-u_m^{t_m}$};
		\draw (\x,\y )  node (center) [above] {$-e$};
		\draw (\x,\y- \vedge)  node (bottom1) [right] {$-v_m^{r_m}$};
		\draw (\x,\y- \vedge - \dashedvedge)  node (bottom2) [right] {$-v_m^{2}$};
		\draw (\x,\y- 2*\vedge - \dashedvedge)  node (bottom3) [right] {$-v_m^{1}$};
		
		\end{tikzpicture}
		\caption{$\tilde \Gamma_{\vec{\tilde q}, \vec {\tilde p}, m, e}$}\label{curve-resolution}
	\end{center}
\end{figure}
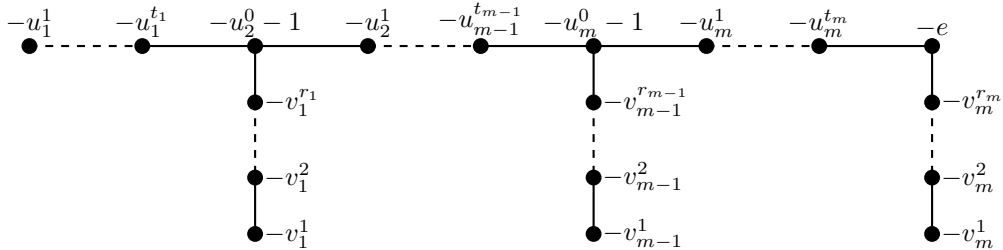
where the right-most vertex in the top row has weight $-e$, and the other weights satisfy: $u_i^0, v_i^0 \geq 1$ and $u_i^j, v_i^j \geq 2$ for $j > 0$, and are uniquely determined from the continued fractions: 
\begin{gather} \label{continued-fractions}
\frac{\tilde p_i}{q'_i} = u_i^0 - \cfrac{1}{u_i^1 - \cfrac{1}{\ddots - \frac{1}{u_i^{t_i}}}},\quad 
\frac{q'_i}{\tilde p_i} = v_i^0 - \cfrac{1}{v_i^1 - \cfrac{1}{\ddots - \frac{1}{v_i^{r_i}}}},\quad 
\text{where}\ q'_i := 
\begin{cases}
\tilde q_1 								& \text{if}\ i = 1\\
\tilde q_i - \tilde q_{i-1}\tilde p_i	& \text{otherwise.}
\end{cases}
\end{gather}

\begin{rem}
$\tilde \Gamma_{\vec{\tilde q}, \vec {\tilde p}, m, 1}$ is the weighted dual graph of the exceptional divisor of the minimal resolution of an irreducible plane curve singularity with Puiseux pairs $(\tilde q_1, \tilde p_1), \ldots, (\tilde q_m, \tilde p_m)$ (see, e.g.\ \cite[Section 2.2]{mendris-nemethi}).
\end{rem}

\begin{thm}[{\cite[Proposition 4.2, Corollary 6.3]{sub2-1}}] \label{primitive-resolution-graph}
Let $\bar X$ be a primitive normal compactification of $X := \spec \cc[x,y] \cong \cc^2$. 
\begin{enumerate}
\item If $\bar X$ is nonsingular, then $\bar X \cong \pp^2$.
\item \label{graph-assertion} Assume $\bar X$ is singular. Let $\tilde \phi_\delta(x,\xi)$ be the generic descending Puiseux series (Definition \ref{generic-descending-defn}) associated to $E^* := \bar X \setminus X$ and $(q_1,p_1), \ldots, (q_{l+1}, p_{l+1})$ be the {\em formal Puiseux pairs} of $\tilde \phi(x,\xi)$ (Definition \ref{formal-defn}). Define $(\tilde q_i, \tilde p_i) := (p_1 \cdots p_i - q_i, p_i)$, $1 \leq i \leq l+1$. 
\begin{enumerate}
\item \label{normal-form} After a (polynomial) change of coordinates of $\cc^2$ if necessary, we may assume that $q_1 < p_1$ and either $l= 0$ or $q_1 > 1$.
\item \label{case:graph1} Assume \eqref{normal-form} holds. If $p_{l+1} > 1$, then the augmented and marked dual graph of the minimal $\pp^2$-dominating resolution of singularities of $\bar X$ is as in figure \ref{fig:graph1}, where the strict transform of the curve at infinity on $\pp^2$ (resp.\ $\bar X$) is marked by $L$ (resp.\ $E^*$).
\item \label{case:graph2} Assume \eqref{normal-form} holds. If $p_{l+1} = 1$, then ($l \geq 1$, and) the augmented and marked dual graph of the minimal $\pp^2$-dominating resolution of singularities of $\bar X$ is as in figure \ref{fig:graph2}, where the strict transform of the curve at infinity on $\pp^2$ (resp.\ $\bar X$) is marked by $L$ (resp.\ $E^*$).
\end{enumerate}
\item \label{converse-graph} Conversely, let $0 \leq l$, and $(q_1,p_1), \ldots, (q_{l+1}, p_{l+1})$ be pairs of integers such that
\begin{enumerate}
\item $p_k \geq 2$, $1 \leq k \leq l$.
\item $p_{l+1} \geq 1$.
\item $\tilde q_k := p_1 \cdots p_k - q_k > 0$, $1 \leq k \leq l+1$.
\item $\gcd(p_k, q_k) = 1$, $1 \leq k \leq l+1$.
\end{enumerate}
Assume moreover that \eqref{normal-form} holds, i.e.\ either $l = 0$ or $q_{l+1} > 1$. Define $\omega_0, \ldots, \omega_{l+1}$ as in \eqref{omega-defn}. Let
\begin{align}
\Gamma_{\vec p, \vec q} := 
\begin{cases}
\text{the graph from Figure \ref{fig:graph1}} & \text{if}\ p_{l+1} > 1, \\
\text{the graph from Figure \ref{fig:graph2}} & \text{if}\ p_{l+1} = 1.
\end{cases}
\label{Gamma-p-q}
\end{align}
Then $\Gamma_{\vec p,\vec q}$ is the augmented and marked dual graph of the minimal $\pp^2$-dominating resolution of singularities of a primitive normal analytic compactification of $\cc^2$ iff $\omega_{l+1} > 0$.
\end{enumerate}
\end{thm} 

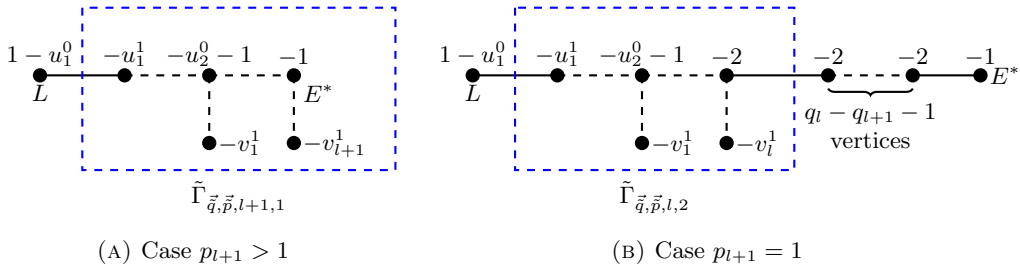
\begin{figure}[htp]
\newcommand{\curveresolutionsmallblock}[5]{
	\pgfmathsetmacro\edge{1.25}
	\pgfmathsetmacro\vedge{1}
	\pgfmathsetmacro\x{0}
	\pgfmathsetmacro\y{0}
	
	\draw[thick] (\x - \edge,\y) -- (\x,\y);
	\draw[thick, dashed] (\x,\y) -- (\x+\edge,\y);
	\draw[thick, dashed] (\x + \edge,\y) -- (\x+\edge,\y - \vedge);
	\draw[thick, dashed] (\x + \edge, \y) -- (\x + 2*\edge, \y);
	\draw[thick, dashed] (\x + 2*\edge,\y) -- (\x+ 2*\edge,\y - \vedge);
	
	\fill[black] (\x - \edge, \y) circle (3pt);
	\fill[black] (\x, \y) circle (3pt);
	\fill[black] (\x + \edge, \y) circle (3pt);
	\fill[black] (\x + \edge, \y - \vedge) circle (3pt);
	\fill[black] (\x + 2*\edge, \y) circle (3pt);
	\fill[black] (\x + 2*\edge, \y - \vedge) circle (3pt);
	
	\draw (\x- \edge,\y )  node (e0up) [above] {$1 - u_1^0$};
	\draw (\x- \edge,\y )  node (e0down) [below] {$L$};
	\draw (\x,\y )  node (e1up) [above] {$-u_1^1$};
	\draw (\x+ \edge,\y )  node (middleup) [above] {$-u_2^0-1$};
	\draw (\x+ \edge,\y - \vedge)  node (middlebottom) [right] {$-v_1^1$}; 	
	\draw (\x+ 2*\edge,\y )  node (estarup) [above] {#1};
	\draw (\x+ 2*\edge,\y )  node (estardown) [below right] {#2};
	\draw (\x+ 2*\edge, \y - \vedge)  node (rightbottom) [right] {$-v_{#3}^1$};
	
	\draw[thick, dashed, blue] (\x - 0.5*\edge, \y + \vedge) rectangle (\x + 2*\edge + #5, \y - 1.4*\vedge);
	
	\draw (\x + 0.75*\edge + 0.5*#5, \y - 1.8*\vedge) node (Gamma) {#4};
}
\centering
\subcaptionbox[]{%
   Case $p_{l+1} > 1$%
    \label{fig:graph1}%
}
[%
    0.4\textwidth 
]%
{%
	\begin{tikzpicture}[scale=.9, font = \small] 	
		\curveresolutionsmallblock{$-1$}{$E^*$}{l+1}{$\tilde \Gamma_{\vec{\tilde q}, \vec{\tilde p}, l+1, 1}$}{1.5}
	\end{tikzpicture}
}
\subcaptionbox[]{%
	Case $p_{l+1} = 1$%
	\label{fig:graph2}
}
[%
    0.5\textwidth 
]%
{%
	\begin{tikzpicture}[scale=.9, font = \small] 	
		\curveresolutionsmallblock{$-2$}{}{l}{$\tilde \Gamma_{\vec{\tilde q}, \vec{\tilde p}, l, 2}$}{1}
		\pgfmathsetmacro\edge{1.25}
		\pgfmathsetmacro\extralen{0.5}
		\begin{scope}[shift={(2*\edge + \extralen,0)}]
		\pgfmathsetmacro\sedge{1}
		\pgfmathsetmacro\dashedge{1.25}
		\draw[thick] (-\extralen, 0) -- (\sedge, 0);
		\draw[thick, dashed] (\sedge, 0) -- (\sedge + \dashedge, 0);
		\draw[thick] (\sedge + \dashedge, 0) -- (2*\sedge + \dashedge, 0);
		
		\fill[black] (\sedge, 0) circle (3pt);
		\fill[black] (\sedge + \dashedge, 0) circle (3pt);
		\fill[black] (2*\sedge + \dashedge, 0) circle (3pt);
		
		\draw (\sedge,0 )  node (firstup) [above] {$-2$};
		\draw (\sedge + \dashedge,0 )  node (secondup) [above] {$-2$};
		\draw (2*\sedge + \dashedge,0 )  node (lastup) [above] {$-1$};
		\draw (2*\sedge + \dashedge,0 )  node (lastright) [right] {$E^*$};
		
		\draw [thick, decoration={brace, mirror, raise=5pt},decorate] (\sedge,0) -- (\sedge + \dashedge,0);
		\draw (\sedge + 0.5*\dashedge,-0.75) node [text width= 1.75cm, align = center] (extranodes) {$q_l - q_{l+1} -1$ vertices};
		\end{scope}	
		\end{tikzpicture}
}
\caption{Augmented and marked dual graph for the minimal $\pp^2$-dominating resolutions of singularities of primitive normal analytic compactifications of $\cc^2$ }
\label{fig:primitive-resolution-graph}
\end{figure}

\begin{rem} \label{essential-value-remark}
Let $\bar X$ be a primitive normal analytic compactifications of $\cc^2 = \spec \cc[x,y]$ and $\Gamma$ be the augmented and marked dual graph for the minimal $\pp^2$-dominating resolution of singularities of $\bar X$. Theorem \ref{primitive-resolution-graph} and identity \eqref{continued-fractions} imply that $\Gamma$ determines and is determined by the formal Puiseux pairs of the generic descending Puiseux series associated to the curve $E^*$ at infinity on $\bar X$. Let $\delta$ be the semidegree on $\cc[x,y]$ corresponding to $E^*$. Proposition \ref{essential-value-prop} then implies that the $\delta$-value of essential key forms of $\delta$ are also uniquely determined by $\Gamma$; we call these the {\em essential key values} of $\Gamma$.
\end{rem}

%
%
%
%

\section{Main results} \label{result-section}
Consider the set up of Question \ref{contract-question}. Assume $N = 1$. 
Choose coordinates $(x,y)$ on $\pp^2\setminus L$. Let $\delta$ be the semidegree on $\cc[x,y]$ associated to $E_1$ (i.e.\ $\delta$ is the order of pole along $E_1$) and let $g_0, \ldots, g_{n+1} \in \cc[x,x^{-1},y]$ be the key forms of $\delta$. 

\begin{thm}[Answer to $N=1$ case of Question \ref{contract-question}] \label{effective-answer}
The following are equivalent:
\begin{enumerate}
\item \label{algebraic-Y'} $Y'$ is algebraic.
\item \label{all-polynomials} $g_j$ is a polynomial, $0 \leq j \leq n+1$.
\item \label{last-polynomial} $g_{n+1}$ is a polynomial.
\end{enumerate}
Moreover, if any of the above conditions holds, then $Y'$ is isomorphic to the closure of the image of $\cc^2$ in the weighted projective variety $\pp^{n+2}(1,\delta(g_0), \ldots, \delta(g_{n+1}))$ under the mapping $(x,y) \mapsto [1:g_0:\cdots:g_{n+1}]$. 
\end{thm}

\begin{rem} \label{analytic-contractibility}
Note that to ask Question \ref{contract-question} we need to determine if the given curve $E'$ is analytically contractible. We would like to point out that in addition to the direct application of Grauert's criterion, the contractibility of $E'$ can be determined in terms of the semidegrees associated to $E_1, \ldots, E_N$ \cite[Theorem 1.4]{sub2-1}. In particular, in the $N=1$ case, $E'$ of Question \ref{contract-question} is analytically contractible iff $\delta(g_{n+1}) > 0$ (where $\delta$ and $g_{n+1}$ are as above).
\end{rem}

We now state the correspondence between primitive normal algebraic compactifications of $\cc^2$ and algebraic curves in $\cc^2$ with one place at infinity:

\begin{thm} \label{geometric-answer}
	Let $\bar X$ be a primitive normal analytic compactification of $\cc^2$. Let $P \in \bar X\setminus \cc^2$ be a center of a $\pp^2$-infinity on $\bar X$ (Remark-Definition \ref{P^2-center}). Then the following are equivalent:
	\begin{enumerate}
		\item \label{algebraic-X-bar} $\bar X$ is algebraic.
		\item \label{one-place-curve} There is an algebraic curve $C$ in $\cc^2$ with {\em one place at infinity} such that $P$ is not on the closure of $C$ in $\bar X$.
	\end{enumerate}
	Let $\delta$ be the semidegree on $\cc[x,y]$ corresponding to the curve at infinity on $\bar X$, and $g_0, \ldots, g_{n+1} \in \cc[x,x^{-1},y]$ be the sequence of key forms of $\delta$. If either \eqref{algebraic-X-bar} or \eqref{one-place-curve} is true, then $g_{n+1}$ is a polynomial, and defines a curve $C$ as in \eqref{one-place-curve}.
\end{thm}

Now we come to the question of characterization of augmented and marked dual graphs of the resolution of singularities of primitive normal analytic compactifications of $\cc^2$. For a primitive normal analytic compactification $\bar X$ of $\cc^2$, let $\Gamma_{\bar X}$ be the {\em augmented and marked} dual graph (from Thorem \ref{primitive-resolution-graph}) associated to the minimal $\pp^2$-dominating resolution of singularities of $\bar X$. Let $\scrG$ be the collection of $\Gamma_{\bar X}$ as $\bar X$ varies over all primitive normal analytic compactifications of $\cc^2$; note that assertions \eqref{graph-assertion} and \eqref{converse-graph} of Theorem \ref{primitive-resolution-graph} gives a complete description of $\scrG$. Pick $\Gamma \in \scrG$. Let $(q_1, p_1), \ldots, (q_{l+1}, p_{l+1})$ be the formal Puiseux pairs, and $\omega_0, \ldots, \omega_{l+1}$ be the sequence of {\em essential key values} of $\Gamma$ (Remark \ref{essential-value-remark}). Fix $k$, $1 \leq k \leq l$. The {\em semigroup conditions} for $k$ are:
\begin{gather}
p_k \omega_k \in \zz_{\geq 0} \langle \omega_0, \ldots, \omega_{k-1}\rangle .  \tag{S1-k} \label{semigroup-criterion-1} \\
(\omega_{k+1}, p_k \omega_k) \cap \zz\langle \omega_0, \ldots, \omega_k \rangle = (\omega_{k+1}, p_k \omega_k) \cap \zz_{\geq 0}\langle \omega_0, \ldots, \omega_k \rangle,\ \tag{S2-k} \label{semigroup-criterion-2}
\end{gather}
where $(\omega_{k+1}, p_k \omega_k): = \{a \in \rr: \omega_{k+1} < a < p_k \omega_k\}$ and $\zz_{\geq 0}\langle \omega_0, \ldots, \omega_k \rangle$ (respectively, $\zz \langle \omega_0, \ldots, \omega_k \rangle$) denotes the semigroup (respectively, group) generated by linear combinations of $\omega_0, \ldots, \omega_k$ with non-negative integer (respectively, integer) coefficients. 
\begin{thm} \label{graph-thm}
\mbox{}
\begin{enumerate}
\item \label{algebraic-graph} $\Gamma = \Gamma_{\bar X}$ for some primitive normal {\em algebraic} compactification $\bar X$ of $\cc^2$ iff the semigroup conditions \eqref{semigroup-criterion-1} hold for all $k$, $1 \leq k \leq l$.
\item \label{non-algebraic-graph} $\Gamma = \Gamma_{\bar X}$ for some primitive normal {\em non-algebraic} compactification $\bar X$ of $\cc^2$ iff either \eqref{semigroup-criterion-1} or \eqref{semigroup-criterion-2} {\em fails} for some $k$, $1 \leq k \leq l$.
\end{enumerate}
\end{thm}

\begin{remexample}
Note that if \eqref{semigroup-criterion-1} holds for all $k$, $1 \leq k \leq l$, but \eqref{semigroup-criterion-2} {\em fails} for some $k$, $1 \leq k \leq l$, then Theorem \ref{graph-thm} implies that there exist primitive normal analytic compactifications $\bar X_1,\bar X_2$ of $\cc^2$ such that $\bar X_1$ is algebraic, $\bar X_2$ is {\em not} algebraic, and $\Gamma = \Gamma_{\bar X_1} = \Gamma_{\bar X_2}$. Indeed, that is precisely what happens in the set up of Section \ref{non-example-section}: let $\Gamma$ be the augmented and marked dual graph corresponding to the minimal $\pp^2$-dominating resolution of singularities of $Y'_i$'s (Figure \ref{fig:non-graph-example}).
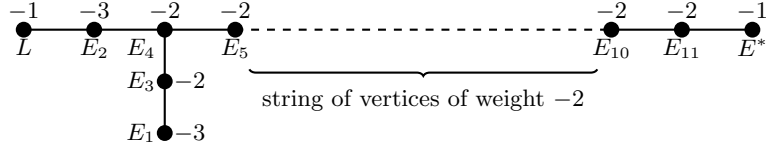
\begin{figure}[htp]

\centering
\begin{tikzpicture}[scale=1, font = \small] 

	\pgfmathsetmacro\factor{1.25}
 	\pgfmathsetmacro\dashedge{4*\factor}	
 	\pgfmathsetmacro\edge{.75*\factor}
 	\pgfmathsetmacro\vedge{.55*\factor}

 	\draw[thick] (-2*\edge,0) -- (\edge,0);
 	\draw[thick] (0,0) -- (0,-2*\vedge);
 	\draw[thick, dashed] (\edge,0) -- (\edge + \dashedge,0);
 	\draw[thick] (\edge + \dashedge,0) -- (2*\edge + \dashedge,0);
 	
 	\fill[black] ( - 2*\edge, 0) circle (3pt);
 	\fill[black] (-\edge, 0) circle (3pt);
 	\fill[black] (0, 0) circle (3pt);
 	\fill[black] (0, -\vedge) circle (3pt);
 	\fill[black] (0, - 2*\vedge) circle (3pt);
 	\fill[black] (\edge, 0) circle (3pt);
 	\fill[black] (\edge + \dashedge, 0) circle (3pt);
 	\fill[black] (2*\edge + \dashedge, 0) circle (3pt);
 	
 	\draw (-2*\edge,0)  node (e0up) [above] {$-1$};
 	\draw (-2*\edge,0 )  node (e0down) [below] {$L$};
 	\draw (-\edge,0 )  node (e2up) [above] {$-3$};
 	\draw (-\edge,0 )  node [below] {$E_2$};
 	\draw (0,0 )  node (e4up) [above] {$-2$};	
 	\draw (0,0 )  node [below left] {$E_4$};
 	\draw (0,-\vedge )  node (down1) [right] {$-2$};
 	\draw (0,-\vedge )  node [left] {$E_3$};
 	\draw (0, -2*\vedge)  node (down2) [right] {$-3$};
 	\draw (0,-2*\vedge )  node [left] {$E_1$};
 	\draw (\edge,0)  node (e+1-up) [above] {$-2$};
 	\draw (\edge,0)  node [below] {$E_5$};
 	\draw (\edge + \dashedge,0)  node (e-last-1-up) [above] {$-2$};
 	\draw (\edge + \dashedge,0)  node [below] {$E_{10}$};
 	\draw (2*\edge + \dashedge,0)  node (e-last-up) [above] {$-2$};
 	\draw (2*\edge + \dashedge,0)  node [below] {$E_{11}$};
 	
 	\pgfmathsetmacro\factorr{.2}
 	\draw [thick, decoration={brace, mirror, raise=15pt},decorate] (\edge + \edge*\factorr,0) -- (\edge + \dashedge - \edge*\factorr,0);
 	\draw (\edge + 0.5*\dashedge,-\edge) node [text width= 4.5cm, align = center] (extranodes) {string of vertices of weight $-2$};
 	
 	\draw[thick] (2*\edge + \dashedge,0) -- (3*\edge + \dashedge,0);
 	\fill[black] (3*\edge + \dashedge, 0) circle (3pt);
 	\draw (3*\edge + \dashedge,0)  node (e-last-up) [above] {$-1$};
 	\draw (3*\edge + \dashedge,0)  node [below] {$E^*$};
 			 	
\end{tikzpicture}
\caption{Augmented and marked dual graph of the minimal $\pp^2$-dominating resolution of singularities of $Y'_i$ from Section \ref{non-example-section}}
\label{fig:non-graph-example}
\end{figure}
It follows from \eqref{non-phi-delta} that the formal Puiseux pairs associated to $\Gamma$ are $(2,5), (-6,1)$; in particular $l=1$. Example \ref{non-essential-example} implies that the sequence of essential key values of $\Gamma$ is $(5,2,2)$. It is straightforward to verify that \eqref{semigroup-criterion-1} is satisfied for $k =1$. On the other hand, 
$$3 \in (2,10) \cap \zz \langle 5, 2 \rangle \setminus  \zz_{\geq 0}\langle 5, 2 \rangle$$
so that \eqref{semigroup-criterion-2} is violated for $k = 1$. This implies that $\Gamma$ corresponds to both algebraic and non-algebraic normal compactifications of $\cc^2$, as we have already seen in Section \ref{non-example-section}.
\end{remexample}

\begin{remexample}
The following are straightforward corollaries of Theorem \ref{graph-thm} and the fact (which is a special case of \cite[Proposition 2.1]{herzog}) that if $p,q$ are relatively prime positive integers, then the greatest integer not belonging to $\zz_{\geq 0}\langle p, q\rangle$ is $ pq - p - q$. 
\begin{enumerate}
\item Pick relatively prime positive integers $p,q$ such that $p > q$. Then $\Gamma_{p,q}$ (defined as in \eqref{Gamma-p-q}) corresponds to only algebraic compactifications of $\cc^2$.

\item Pick integers $p,q,r$ such that $p,q$ are relatively prime, $p > q > 0$ and $q > r$. Set $l := 1$, $(q_1,p_1) := (q,p)$, $(q_2,p_2) := (r,1)$. Then \eqref{omega-defn} implies that $\omega_0 = p$, $\omega_1 = q$ and $\omega_2 = (p-1)q + r$. Assertion \eqref{converse-graph} of Theorem \ref{primitive-resolution-graph} therefore implies that $\Gamma_{\vec p,\vec q}$ corresponds to a compactification of $\cc^2$ iff $(p-1)q + r > 0$. So assume $q > r >- (p-1)q$.
\begin{enumerate}
\item If $r \geq -p$, then $\Gamma_{\vec p,\vec q}$ corresponds to only algebraic compactifications of $\cc^2$.
\item If $-p > r > - (p-1)q$, then $\Gamma_{\vec p,\vec q}$ corresponds to both algebraic and non-algebraic compactifications of $\cc^2$.
\end{enumerate}


\item Let $p_1,q_1,p_2$ be integers such that $p_1 > q_1 > 1$, $p_2 \geq 2$, $p_1$ is relatively prime to $q_1$, and $p_2$ is relatively prime to $p_1q_1 - p_1 - q_1$. Set $q_2 := p_1q_1 - p_1 - q_1 - q_1(p_1-1)p_2$, $q_3 := q_2 - 1$ and $p_3 := 1$. In this case $\omega_0 = p_1p_2$, $\omega_1 = q_1p_2$, $\omega_2 = p_1q_1 - p_1 - q_1$ and $\omega_3 = p_2\omega_2 - 1$. It follows that \eqref{semigroup-criterion-1} fails for $k = 2$ and therefore $\Gamma_{\vec p,\vec q}$ corresponds to {\em only non-algebraic} compactifications of $\cc^2$.
\end{enumerate}
\end{remexample}

Finally we formulate our answer to $N=1$ case of Question \ref{cousin-question}. Let $O \in L_\infty := \pp^2\setminus \cc^2$. Let $(u, v)$ be coordinates near $O$, $\psi(u)$ be a Puiseux series in $u$, and $r$ be a positive rational number. After a change of coordinates near $O$ if necessary, we may assume that the coordinate of $O$ is $(0,0)$ with respect to $(u,v)$-coordinates, and $(x,y):= (1/u, v/u)$ is a system of coordinates on $\pp^2 \setminus L_\infty \cong \cc^2$. Let
\begin{align*}
\phi(x) := x\psi(1/x)
\end{align*}
Note that $\phi(x)$ is a descending Puiseux series in $x$. Let $\xi$ be an indeterminate, and define, following Notation \ref{descending-notation},
$$\tilde \phi(x, \xi) := [\phi(x)]_{> 1-r} + \xi x^{1-r}$$
Let $\delta$ be the semidegree on $\cc[x,y] $ with generic descending Puiseux series $\tilde \phi$, and let $g_0, \ldots, g_{n+1} \in \cc[x,x^{-1}, y]$ be the key forms of $\delta$ (see Algorithm \ref{key-algorithm} for the algorithm to determine key forms of $\delta$ from $\tilde \phi$). 

\begin{thm}[Answer to $N=1$ case of Question \ref{cousin-question}] \label{cousin-answer}
The following are equivalent:
\begin{enumerate}
\item \label{cousin-polynomial} There exists a polynomial $f \in \cc[x,y]$ such that for each analytic branch $C$ of the curve $f = 0$ at infinity, 
\begin{itemize}
	\item $C$ intersects $L_\infty$ at $O$,
	\item  $C$ has a Puiseux expansion $v = \theta(u)$ at $O$ such that $\ord_{u}(\theta - \psi) \geq r$. 
\end{itemize}

\item \label{all-cousinomials} $g_j$ is a polynomial, $0 \leq j \leq n+1$.
\item \label{last-cousinomial} $g_{n+1}$ is a polynomial.
\end{enumerate}
Moreover, if any of the above conditions holds, then $g_{n+1}$ satisfies the properties of $f$ from condition \eqref{cousin-polynomial}.
\end{thm}

\section{Background II: notions required for the proof} \label{secondground}
In this section we collect more background material we use in the proof of the results stated in Section \ref{result-section}. 

\subsection{Key forms from descending Puiseux series} \label{essential-section}
Let $\delta$ be a semidegree on $\cc[x,y]$ such that $\delta(x) > 0$. Assume the generic descending Puisuex series for $\delta$ is 
\begin{align*}
\tilde \phi_\delta(x,\xi) 
	&:= \sum_{j=1}^{k'_0} a_{0j}x^{q_{0j}} + a_1 x^{\frac{q_1}{p_1}} + \cdots
		 +  a_2 x^{\frac{q_2}{p_1p_2}} + \cdots + a_{l} x^{\frac{q_{l}}{p_1p_2 \cdots p_{l}}} 
		 +  \xi x^{\frac{q_{l + 1}}{p_1p_2 \cdots p_{l+1}}} 
\end{align*}
where $q_{01} > \cdots > q_{0k'_0}$ are integers greater than $q_1/p_1$, $k'_0 \geq 0$, and $(q_1, p_1), \ldots, (q_{l + 1}, p_{l+1})$ are the {\em formal} Puiseux pairs of $\tilde \phi_\delta$ (Definition \ref{formal-defn}). Let $g_0 = x, g_1 = y, \ldots, g_{n+1} \in \cc[x,x^{-1},y]$ be the key forms of $\delta$. Recall from Proposition \ref{essential-value-prop} that precisely $l+2$ of the key forms of $\delta$ are {\em essential}. Let $0 = j_0 < \cdots < j_{l+1} = n+1$ be the subsequence $(0, \ldots, n)$ consisting of indices of essential key forms of $\delta$.

\begin{algorithm}[Construction of key forms from descending Puiseux series, cf.\ the algorithm in \cite{lenny}] \label{key-algorithm}
\mbox{}
\paragraph{1. Base step:} Set $j_0 := 0$, $g_0 := x,$, $g_1 := y$. Also define $p_0 := 1$.

Now assume
\begin{enumerate}[(i)]
\item $g_0, \ldots, g_s$ have been calculated, $s \geq 1$. 
\item $j_0, \ldots, j_k$ have been calculated, $k \geq 0$. 
\item $j_k < s \leq j_{k+1}$.
\end{enumerate}

\paragraph{2. Inductive step for $(s,k)$:} Let 
\begin{align*}
\tilde \omega_s
	&:= \deg_x \left(g_s|_{y = \tilde \phi_\delta(x,\xi)}\right) \\
\tilde c_s
	&:= \text{coefficient of $x^{\tilde \omega_s}$ in $g_s|_{y = \tilde \phi_\delta(x,\xi)}$}
\end{align*}
\paragraph{Case 2.1:} If $\tilde c_s \in \cc[\xi]\setminus \cc$, then set $n := s-1$, $j_{k+1} := s$, and stop the process.\\

\paragraph{Case 2.2:} Otherwise if $\tilde \omega_s \in \frac{1}{p_0 \cdots p_k}\zz$, then there are unique integers $\beta^s_0, \cdots, \beta^s_k$ and unique $c \in \cc^*$ such that 
\begin{compactenum}
\item $0 \leq \beta^s_i < p_i$ for $1 \leq i \leq k$,
\item $\sum_{i=0}^k \beta^s_i \tilde \omega_{j_i} = \tilde \omega_s$, and
\item the coefficient of $x^{\tilde \omega_s}$ in $cg_{j_0}^{\beta^s_0} \cdots g_{j_k}^{\beta^s_k}|_{y = \tilde \phi_\delta(x,\xi)}$ is $\tilde c_s$.
\end{compactenum}
Then set $g_{s+1} := g_s - cg_{j_0}^{\beta^s_0} \cdots g_{j_k}^{\beta^s_k}$, and repeat Inductive step for $(s+1,k)$. \\

\paragraph{Case 2.3:} Otherwise $\tilde \omega_s \in \frac{1}{p_0 \cdots p_{k+1}}\zz\setminus \frac{1}{p_0 \cdots p_k}\zz$, and there are unique integers $\beta^s_0, \cdots, \beta^s_k$ and unique $c \in \cc^*$ such that 
\begin{compactenum}
\item $0 \leq \beta^s_i < p_i$ for $1 \leq i \leq k$,
\item $\sum_{i=0}^k \beta^s_i \tilde \omega_{j_i} = p_{k+1}\tilde \omega_s$, and
\item the coefficient of $x^{\tilde \omega_s}$ in $cg_{j_0}^{\beta^s_0} \cdots g_{j_k}^{\beta^s_k}|_{y = \tilde \phi_\delta(x,\xi)}$ is $(\tilde c_s)^{p_{k+1}}$.
\end{compactenum}
Then set $j_{k+1} := s$, $g_{s+1} := g_s^{p_{k+1}} - cg_{j_0}^{\beta^s_0} \cdots g_{j_k}^{\beta^s_k}$, and repeat Inductive step for $(s+1,k+1)$. 
\end{algorithm}

\begin{example} \label{essential-example}
Let $\tilde \phi_\delta(x,\xi) := x^3 + x^2+ x^{5/3} + x + x^{-13/6} + x^{-7/3} + \xi x^{-8/3}$. The formal Puiseux pairs of $\tilde \phi_\delta$ are $(5,3), (-13,2), (-16,1)$. We compute the key forms of $\delta$ following Algorithm \ref{key-algorithm}: by definition we have $g_0 = x$, $g_1 = y$, $j_0 = 0$. Since the exponents of $x$ in the first two terms of $\tilde \psi_\delta$ are integers, subsequent applications of Case 2.2 of Algorithm \ref{key-algorithm} implies that the next two key forms are $g_2 = y - x^3$ and $g_3 = y - x^3 - x^2$. Note that 
\begin{align}
	g_3|_{y = \tilde \psi_\delta} &= x^{5/3} + x + x^{-13/6} + x^{-7/3} + \xi x^{-8/3},\label{g_3-substitution}
\end{align}
In the notation of Algorithm \ref{key-algorithm}, we have $\tilde \omega_3 = 5/3 \not\in \zz$. It follows that $j_1 = 3$. Since 
\begin{align}
g_3^3|_{y = \tilde \psi_\delta} 
	&=  x^5 + 3x^{13/3} + 3x^{11/3} + x^3 + 3x^{7/6} + 3x + 3\xi x^{2/3} + \ldt,
	\label{g_3^3-substitution}
\end{align}
(where $\ldt$ denotes terms with smaller degree in $x$), Case 2.3 of Algorithm \ref{key-algorithm} implies that $g_4 = g_3^3 - x^5$. Now note that $13/3 = 1 + 2\cdot(5/3)$ and $11/3 = 2 + 5/3$, so that \eqref{g_3-substitution} and \eqref{g_3^3-substitution} imply that
\begin{align*}
g_4|_{y = \tilde \psi_\delta} 
	&=  3x\left(g_3|_{y = \tilde \psi_\delta} - x - x^{-13/6} - x^{-7/3} - \xi x^{-8/3}\right)^2  
	+ 3x^2\left(g_3|_{y = \tilde \psi_\delta} - x  \right.  \\
	& \qquad  \left.  \mbox{\phantom{$g_3|_{y = \tilde \psi_\delta}$}} - x^{-13/6} - x^{-7/3} - \xi x^{-8/3}\right)
		+ x^3 + 3x^{7/6} + 3x + 3\xi x^{2/3} + \ldt\\
	&= 3xg_3^2|_{y = \tilde \psi_\delta} - 3x^2g_3|_{y = \tilde \psi_\delta} + x^3 + 3x^{7/6} + 3x + 3\xi x^{2/3} + \ldt,
\end{align*}
Repeated applications of Case 2.2 of Algorithm \ref{key-algorithm} then imply that 
\begin{align*}
g_5 &= g_4 - 3xg_3^2\\
g_6 &= g_4 - 3xg_3^2 - 3x^2g_3 \\
g_7 &= g_4 - 3xg_3^2 - 3x^2g_3 - x^3.
\end{align*}
Note that 
\begin{align}
g_7|_{y = \tilde \psi_\delta} 
	&= 3x^{7/6} + 3x + 3\xi x^{2/3} + \ldt \label{g_7-substitution}
\end{align}
Since $\tilde \omega_7 = 7/6 \not\in \frac{1}{3}\zz$, following Case 2.3 of Algorithm \ref{key-algorithm} we have $j_2 = 7$. Since $p_2 = 2$, we compute 
\begin{align*}
g_7^2|_{y = \tilde \psi_\delta} 
	&=  9x^{7/3} + 18x^{13/6} + 18\xi x^{11/6} +  \ldt,
\end{align*}
Since $7/3 = -1 + 2\cdot(5/3) + 0\cdot(7/6)$ and $13/6 = 1 + 0\cdot(5/3) + 7/6$, identities \eqref{g_3-substitution} and \eqref{g_7-substitution} imply that
\begin{align*}
g_7^2|_{y = \tilde \psi_\delta} 
		&= 9x^{-1}\left(g_3|_{y = \tilde \psi_\delta} - x - x^{-13/6} - x^{-5/2} - \xi x^{-8/3}\right)^2  + 6x\left(g_7|_{y = \tilde \psi_\delta}\right. \\
		& \quad \left. \mbox{\phantom{$g_7|_{y = \tilde \psi_\delta}$}}  - 3x - 3\xi x^{2/3} - \ldt\right) + 18\xi x^{11/6} + \ldt \\
		& = 9x^{-1} g_3^2|_{y = \tilde \psi_\delta}  + 6xg_7|_{y = \tilde \psi_\delta}  -18x^2+ 18\xi x^{11/6} + \ldt,
\end{align*}
Cases 2.3 and 2.2 of Algorithm \ref{key-algorithm} then imply that the next key forms are 
\begin{align*}
g_8 &= g_7^2 - 9x^{-1} g_3^2 \\
g_9 &= g_7^2 - 9x^{-1} g_3^2 - 6xg_7 \\
g_{10} &= g_7^2 - 9x^{-1} g_3^2 - 6xg_7 + 18x^2
\end{align*}
Since 
\begin{align*}
g_{10}|_{y = \tilde \psi_\delta} = 18\xi x^{11/6} + \ldt,
\end{align*}
Case 2.1 of Algorithm \ref{key-algorithm} implies that $g_{10}$ is the last key form of $\delta$, and $n = 9$, $j_3 = 10$. In particular, note that there are precisely $4$ essential key forms (namely $g_0, g_3, g_7,g_{10}$) of $\delta$, as predicted by Proposition \ref{essential-value-prop}.
\end{example}

The assertions of the following proposition are straightforward implications of Algorithm \ref{key-algorithm}. 

\begin{prop} \label{key-properties}
Let $\delta$ be a semidegree on $\cc[x,y]$ such that $\delta(x) > 0$. Let $g_0, \ldots, g_{n+1}$ be key forms and $\tilde \phi_\delta(x,\xi) := \phi_\delta(x) + \xi x^{r_\delta}$ be the generic descending Puiseux series of $\delta$, 
\begin{enumerate}
\item \label{truncassertion} Let $n_* \leq n$ and let $\delta_*$ be the unique semidegree such that the key forms of $\delta_*$ are $g_0, \ldots, g_{n_*+1}$ and $\delta^*(g_j) = \delta(g_j)/e$, $0 \leq j \leq n_*+1$, where $e := \gcd\left(\delta(g_0), \ldots, \delta(g_{n_*+1})\right)$. Then $\delta_*$ has a generic descending Puiseux series of the form 
$$\tilde \phi_{\delta_*}(x,\xi) = \phi_*(x) + \xi x^{r_*},$$
where 
\begin{compactenum}
\item $r_* \geq r_\delta$, and
\item $\phi_*(x) = [\phi_\delta(x)]_{> r_*}$.
\end{compactenum}
\item \label{key-semigroup} Let $\alpha_i$, $1 \leq i \leq n$, be the smallest positive integer such that $\alpha_i\delta(g_i)$ is in the (abelian) group generated by $\delta(g_0), \ldots, \delta(g_{i-1})$.  Fix $m$, $0 \leq m \leq n$. Recall that each $g_i$ is an element in $\cc[x,x^{-1},y]$ which is monic in $y$. The following are equivalent:
\begin{enumerate}
\item $g_i$ is a polynomial, $0 \leq i \leq m+1$.
\item For each $i$, $1 \leq i \leq m$, $\alpha_i\delta(g_i)$ is in the semigroup generated by $\delta(g_0), \ldots, \delta(g_{i-1})$.
\end{enumerate}
\end{enumerate}
\end{prop}

\subsection{Degree-like functions and compactifications} \label{degree-like-section}
In this subsection we recall from \cite{sub1} the basic facts of compactifications of affine varieties via {\em degree-like functions}. Recall that $X = \cc^2$ in our notation; however the results in this subsection remains valid if $X$ is an arbitrary affine variety.

\begin{defn} \label{degree-like-defn}
A map $\delta: \cc[X] \setminus \{0\} \to \zz$ is called a {\em degree-like function} if 
\begin{compactenum}
\item \label{deg1} $\delta(f+g) \leq \max\{\delta(f), \delta(g)\}$ for all $f, g \in \cc[X]$, with $<$ in the preceding inequality implying $\delta(f) = \delta(g)$.
\item \label{deg2} $\delta(fg) \leq \delta(f) + \delta(g)$ for all $f, g \in \cc[X]$.
\end{compactenum}
\end{defn}

Every degree-like function $\delta$ on $\cc[X]$ defines an {\em ascending filtration} $\{\scrF^\delta_d\}_{d \geq 0}$ on $\cc[X]$, where $\scrF^\delta_d := \{f \in \cc[X]: \delta(f) \leq d\}$. Define
$$ \cc[X]^\delta := \dsum_{d \geq 0} \scrF^\delta_d, \quad \gr  \cc[X]^\delta := \dsum_{d \geq 0} \scrF^\delta_d/\scrF^\delta_{d-1}.$$

\begin{rem} \label{(f)_d-remark}
For every $f \in \cc[X]$, there are infinitely many `copies' of $f$ in $\cc[X]^\delta$, namely the copy of $f$ in $\scrF^\delta_d$ for {\em each $d \geq \delta(f)$}; we denote the copy of $f$ in $\scrF^\delta_d$ by $(f)_d$. If $t$ is a new indeterminate, then 
$$\cc[X]^\delta \cong \sum_{d \geq 0} \scrF^\delta_d t^d,$$
via the isomorphism $(f)_d \mapsto ft^d$. Note that $t$ corresponds to $(1)_1$ under this isomorphism.
\end{rem}

We say that $\delta$ is {\em finitely-generated} if $\cc[X]^\delta$ is a finitely generated algebra over $\cc$ and that $\delta$ is {\em projective} if in addition $\scrF^\delta_0 = \cc$. The motivation for the terminology comes from the following straightforward

\begin{prop}[{\cite[Proposition 2.8]{sub1}}] \label{basic-prop}
If $\delta$ is a projective degree-like function on $\cc[X]$, then $\xdelta := \proj \cc[X]^\delta$ is a projective compactification of $X$. The {\em hypersurface at infinity} $\xdelta_\infty := \xdelta \setminus X$ is the zero set of the $\qq$-Cartier divisor defined by $(1)_1$ and is isomorphic to $\proj \gr \cc[X]^\delta$. Conversely, if $\bar X$ is any projective compactification of $X$ such that $\bar X \setminus X$ is the support of an effective ample divisor, then there is a projective degree-like function $\delta$ on $\cc[X]$ such that $\xdelta \cong \bar X$.
\end{prop}

\begin{rem} \label{semi-remark}
The {\em semidgree}, which we already defined in Section \ref{divisorial-section}, is a degree-like function which always satisfies property \ref{deg2} with an equality. 
\end{rem}

The following proposition (which is straightforward to prove) is a special case of \cite[Theorem 4.1]{sub1}.

\begin{prop} \label{basic-semi-prop}
Let $\delta$ be a projective degree-like function on $\cc[X]$, and $\xdelta$ be the corresponding projective compactification from Proposition \ref{basic-prop}. Assume $\delta$ is a semidegree. Then $\xdelta$ is a normal variety and $\xdelta_\infty := \xdelta \setminus X$ is an irreducible codimension one subvariety. Moreover, there is a positive integer $d_\delta$ such that $\delta$ agrees with $d_\delta$ times the order of {\em pole} along $\xdelta_\infty$.
\end{prop}



\section{Some preparatory results} \label{prep-section}
In this section we develop some preliminary results to be used in Section \ref{proof-section} for the proofs of our main results. 

\begin{convention}
Let $y_1, \ldots, y_k$ be indeterminates. From now on we write $A_k$, $\tilde A_k$, $R$, $\tilde R$ to denote respectively  $\cc[x,x^{-1}, y_1, \ldots, y_k]$, $\dpsxc[y_1, \ldots, y_k]$, $\cc[x,x^{-1},y]$, $\dpsxc[y]$. Below we frequently deal with maps $A_k \to R$. We always (unless there is a misprint!) use upper-case letters $F,G, \ldots$ for elements in $A_k$ and corresponding lower-case letters $f,g, \ldots$ for their images in $R$. 
\end{convention}
 
\subsection{The `star action' on descending Puiseux series} \label{sec-star}
\begin{defn} \label{star-defn}
	Let $\phi = \sum_j a_j x^{q_j/p} \in \dpsxc$ be a descending Puiseux series with polydromy order $p$ and $r$ be a multiple of $p$. Then for all $c \in \cc$ we define 
	$$c \star_r \phi := \sum_j a_j c^{q_jr/p}x^{q_j/p}.$$  
	For $\Phi = \sum_{\alpha \in \zz_{\geq 0}^k} \phi_\alpha(x)y_1^{\alpha_1} \cdots y_k^{\alpha_k} \in \tilde A_k$, the {\em polydromy order} of $\Phi$ is the lowest common multiple of the polydromy orders of all the non-zero $\phi_\alpha$'s. Let $r$ be a multiple of the polydromy order $\Phi$. Then we define
	$$c \star_r \Phi := \sum_\alpha \left(c \star_r \phi_\alpha\right) y_1^{\alpha_1} \cdots y_k^{\alpha_k}.$$   
\end{defn}

\begin{rem}
	It is straightforward to see that in the case that $c$ is an $r$-th root of unity (and $r$ is a multiple of the polydromy order of $\phi$), $c \star_r \phi$ is a {\em conjugate} of $\phi$ (cf.\ Remark-Notation \ref{f-phi}).
\end{rem}

The following properties of the $\star_r$ operator are straightforward to see:

\begin{lemma} \label{star-lemma}
	\mbox{}
	\begin{enumerate}
		\item \label{different-stars} Let $p$ be the polydromy order of $\Phi \in \tilde A_k$, $d$ and $e$ be positive integers, and $c \in \cc$. Then $c \star_{pde} \Phi = c^e \star_{pd} \Phi = c^{de} \star_p \Phi$.
		\item \label{star-sum-product} Let $\Phi_j = \sum_{j} \phi_{j,\alpha}(x)y_1^{\alpha_1} \cdots y_k^{\alpha_k} \in \tilde A_k$ for $j = 1,2$, and $r$ be a multiple of the polydromy order of each non-zero $\phi_{j,\alpha}$. Then $c \star_r \left(\Phi_1 + \Phi_2\right) = \left(c \star_r \Phi_1\right) + \left(c \star_r \Phi_2\right)$ and $c \star_r \left(\Phi_1\Phi_2\right) = \left(c \star_r \Phi_1\right)\left(c \star_r \Phi_2\right)$. 
		\item \label{star-projection} Let $\pi : \tilde A_k \to \tilde R$ be a $\cc$-algebra homomorphism that sends $x \mapsto x$ and $y_j \mapsto f_j \in R$ for $1 \leq j \leq k$. Let $\Phi = \sum_{\alpha} \phi_\alpha(x)y_1^{\alpha_1} \cdots y_k^{\alpha_k} \in \tilde A_k$, $r$ be a multiple of the polydromy order of each non-zero $\phi_\alpha$, and $\mu$ be a (not necessarily primitive) $r$-th root of unity. Then $\pi (\mu \star_r \Phi) = \mu \star_r \pi(\Phi)$. \qed
	\end{enumerate}
\end{lemma} 

\begin{remtation} \label{f-phi}
If $\phi$ is a descending Puiseux series in $x$ with polydromy order $p$, then we write
\begin{align}
f_\phi := \prod_{\parbox{1.5cm}{\scriptsize{$\phi_{j}$ is a con\-ju\-ga\-te of $\phi$}}}\mkern-18mu \left(y - \phi_{j}(x)\right) = \prod_{j=0}^{p-1} \left(y - \zeta^j \star_p \phi(x)\right) \in \tilde R, \label{f-phi-eqn}
\end{align}
where $\zeta$ is a primitive $p$-th root of unity. If $f \in \cc[x,y]$, then its descending Puiseux factorization (Theorem \ref{dpuiseux-factorization}) can be described as follows: there are unique (up to conjugacy) descending Puiseux series $\phi_1, \ldots, \phi_k$, a unique non-negative integer $m$, and $c \in \cc^*$ such that
$$f = cx^m \prod_{i=1}^k f_{\phi_i}$$
Let $(q_1, p_1), \ldots, (q_l, p_l)$ be Puiseux pairs of $\phi$. Set $p_0 :=1$. For each $k$, $0 \leq k \leq l$, we write 
\begin{align}
f_\phi^{(k)} := \prod_{j=0}^{p_0p_1 \cdots p_k-1} \left(y - \zeta^j \star_{p} \phi(x)\right), \label{f-phi-k}
\end{align}
where $\zeta$ is a primitive $(p_1 \cdots p_l)$-th root of unity. Note that $f_\phi^{(l)}  = f_\phi$, and for each $m,n$, $0 \leq m < n \leq l$, 
\begin{align}
f_\phi^{(n)} 
&= \prod_{j=0}^{p_0p_1 \cdots p_{n}-1} \left(y - \zeta^j \star_{p} \phi(x)\right) 
= \prod_{i=0}^{p_{m+1}\cdots p_n-1} \prod_{j=0}^{p_0p_1 \cdots p_m-1} \left(y - \zeta^{ip_0p_1 \cdots p_m +j} \star_{p} \phi(x)\right) \notag \\
&= \prod_{i=0}^{p_{m+1}\cdots p_n-1} \zeta^{ip_0p_1 \cdots p_m} \star_{p} \left(\prod_{j=0}^{p_0p_1 \cdots p_m-1} \left(y - \zeta^{j} \star_{p} \phi(x)\right)\right) \notag\\
&= \prod_{i=0}^{p_{m+1}\cdots p_n-1} \zeta^{ip_0p_1 \cdots p_m} \star_{p} \left(f_\phi^{(m)}\right). \label{f-phi-n-m}
\end{align}
\end{remtation}

\subsection{`Canonical' pre-images of polynomials and their comparison} \label{lifting-section}

\begin{lemma}[`Canonical' pre-images of elements in {$\dpsxc[y]$}] \label{F-Phi}
Let $p_0 := 1, p_1, \ldots, p_{k-1}$ be positive integers, and $\pi:A_k \to R$ be a ring homomorphism which sends $x \mapsto x$ and $y_j \mapsto f_j$, where $f_j$ is monic in $y$ of degree $p_0 \cdots p_{j-1}$, $1 \leq j \leq k$. Then $\pi$ induces a homomorphism $\tilde A_k \to \tilde R$ which we also denote by $\pi$. If $f$ is a non-zero element in $\tilde R$, then there is a unique $F^\pi_f \in \tilde A_k$ such that 
\begin{enumerate}
\item $\pi\left(F^\pi_f \right) = f$ and
\item $\deg_{y_j}(F^\pi_f) < p_{j}$ for all $j$, $1 \leq j \leq k-1$.
\end{enumerate} 
Moreover, if $f$ is monic in $y$ of degree $p_1 \cdots p_{k-1}d$ for some integer $d$, then 
\begin{enumerate}
\addtocounter{enumi}{2}
\item \label{F-Phi-monic} $F^{\pi}_f$ is monic in $y_k$ of degree $d$. 
\item \label{F-Phi-exponents} If the coefficient of $x^\alpha y_1^{\beta_1}\cdots y_k^{\beta_k}$ in $F^\pi_f - y_k^{d}$ is non-zero, then $\sum_{i=1}^j p_0 \cdots p_{i-1}\beta_i < p_1 \cdots p_j$ for all $j$, $1 \leq j \leq k-1$, and $\sum_{i=1}^{k} p_0 \cdots p_{i-1}\beta_i < p_1 \cdots p_{k-1}d$.
\end{enumerate}
Finally, 
\begin{enumerate}
\setcounter{enumi}{4}
\item \label{F-Phi-polynomial} If each of $f, f_1, \ldots, f_k$ is in $\cc[x,y]$ (resp.\ $R$), then $F^\pi_f$ is in $\cc[x,y_1, \ldots, y_k]$ (resp.\ $A_k$). 
\end{enumerate}
\end{lemma}

\begin{proof}
	This follows from an immediate application of \cite[Theorem 2.13]{abhyankar-expansion}.
\end{proof}

Now assume $\delta$ is the semidegree on $\cc[x,y]$ such that $\delta(x) > 0$. Assume the generic descending Puisuex series for $\delta$ is 
\begin{align*}
\tilde \phi_\delta(x,\xi) 
	&:= \phi_\delta(x) + \xi x^{r_\delta}
	= \cdots + a_1 x^{\frac{q_1}{p_1}} + \cdots
		 +  a_2 x^{\frac{q_2}{p_1p_2}} + \cdots + a_{l} x^{\frac{q_{l}}{p_1p_2 \cdots p_{l}}} 
		 +  \xi x^{\frac{q_{l + 1}}{p_1p_2 \cdots p_{l+1}}} \label{tilde-phi-delta}
\end{align*}
where $(q_1, p_1), \ldots, (q_{l + 1}, p_{l+1})$ are the formal Puiseux pairs of $\tilde \phi_\delta$. Let $g_0 = x,g_1 = y, \ldots, g_{n+1} \in R$ be the sequence key forms of $\delta$ and $g_{j_0}, \ldots, g_{j_{l+1}}$ be the subsequence of {\em essential} key forms. Define
\begin{align}
f_k &:= g_{j_k} \\
\omega_k &:= \delta(f_k),
\end{align}
$0 \leq k \leq l+1$. 

\begin{lemma} \label{f-k-lemma}
\begin{align*}
f_{k+1} 
	&= \begin{cases}
			y - \text{a polynomial in}\ x & \text{if}\ k = 0,\\
			f_k^{p_k} - \sum_{i=0}^{m_k}c_{k,i} f_0^{\beta^i_{k,0}} \cdots f_{k}^{\beta^i_{k,k}} & \text{if}\ 1 \leq k \leq l,
		\end{cases} 
\end{align*}
where
\begin{compactenum}
\addtocounter{enumii}{1}
\item $m_k \geq 0$, 
\item $c_{k,i} \in \cc^*$ for all $i$, $0 \leq i \leq m_k$,
\item \label{exponent-bound} $\beta^i_{k,j}$'s are integers such that $0 \leq \beta^i_{k,j} < p_j$ for $1 \leq j \leq k$ and $0 \leq i \leq m_k$,
\item \label{zero} $\beta^0_{k,k} = 0$,
\item \label{decreasing-omega} $p_k \omega_k = \sum_{j = 0}^{k-1}\beta^0_{k,j}\omega_j > \sum_{j = 0}^{k}\beta^1_{k,j}\omega_j > \cdots > \sum_{j = 0}^{k}\beta^{m_k}_{k,j}\omega_j > \omega_{k+1}$.
\end{compactenum}
\end{lemma}

\begin{proof}
The lemma follows from combining property \eqref{next-property} of key forms, assertion \eqref{alpha=p} of Proposition \ref{essential-value-prop}, and the defining property of essential key forms (Definition \ref{essential-defn}). 
\end{proof}
Let $\pi: A_{l+1} \to R$ be the $\cc$-algebra homomorphism which maps $x \mapsto x$ and $y_k \to f_k$, $1 \leq k \leq l+1$, and let $\pi_k := \pi|_{A_k}: A_k \to R$, $1 \leq k \leq l+1$. Let $\omega$ be weighted degree on $A_{l+1}$ corresponding to weights $\omega_0$ for $x$ and $\omega_k$ for $y_k$, $1 \leq k \leq l+1$. We will often abuse the notation and write $\pi$ and $\omega$ respectively for $\pi_k$ and $\omega|_{A_k}$ for each $k$, $1 \leq k \leq l+1$. Define 
\begin{align} \label{F-k}
F_{k+1} &:= \begin{cases}
			y_1 & \text{if}\ k=0,\\
			y_{k}^{p_k} - \sum_{i=0}^{m_k}c_{k,i} x^{\beta^i_{k,0}} y_1^{\beta^i_{k,1}} \cdots y_{k}^{\beta^i_{k,k}}, 
				& \text{if}\ 1 \leq k \leq l,
			\end{cases}
\end{align} 
where $c_{k,i}$'s are $\beta^i_{k,j}$'s are as in Lemma \eqref{f-k-lemma}. Note that $F_1 \in A_1$ and $F_k \in A_{k-1}$ for $2 \leq k \leq l+1$. Moreover, $\pi(F_k) = f_k$ for each $k$, $1 \leq k \leq l+1$. 

\begin{lemma} \label{delta-omega-lemma}
Fix $k$, $1 \leq k \leq l+1$. 
\begin{enumerate}
\item \label{assertion-H-1-2} Let $H_1, H_2$ be two monomials in $A_k$ such that $\deg_{y_j}(H) < p_j$ for all $j$, $1 \leq j \leq k$. Then $\omega(H_1) \neq \omega(H_2)$.
\item \label{assertion-H} Let $H \in A_k$ be such that $\deg_{y_j}(H) < p_j$ for all $j$, $1 \leq j \leq k$. Then $\delta(\pi(H)) = \omega(H)$.
\end{enumerate}

\end{lemma}

\begin{proof}
Assertion \eqref{alpha=p} of Proposition \ref{essential-value-prop} implies that for each $j$, $1 \leq j \leq k$, $p_j$ is the smallest positive integer such that $p_j\omega_j$ is in the group generated by $\omega_0, \ldots, \omega_{j-1}$. This immediately implies assertion \eqref{assertion-H-1-2}. For assertion \eqref{assertion-H}, write $H = \sum_{i \geq 1} H_i$, where $H_i$'s are monomials in $A_k$. By assertion \eqref{assertion-H-1-2} we may assume w.l.o.g.\ that $\omega(H) = \omega(H_1) > \omega(H_2) > \cdots$. Since $\delta(\pi(y_j)) = \delta(f_j) = \omega_j = \omega(y_j)$ for each $j$, $1 \leq j \leq k$, it follows that $\delta(\pi(H_i)) = \omega(H_i)$ for all $i$. It then follows from the definition of degree-like functions (Definition \ref{degree-like-defn}) that $\delta(\pi(H)) = \omega(H_1) = \omega(H)$.
\end{proof}

\begin{lemma} \label{comparison-lemma-1}
For each $k$, $1 \leq k \leq l+1$, define
\begin{align}
r_k &:= \frac{q_k}{p_1p_2 \cdots p_k} \label{r-k}\\
\phi_k &:= [\phi_{\delta}]_{> r_k} \label{phi-k}
\end{align}
Define $f_{\phi_k} \in \tilde R$ as in \eqref{f-phi-eqn}. Also define
\begin{align*}
F_{\phi_k} := \begin{cases}
	F^{\pi_1}_{f_{\phi_1}} \in \tilde A_1 & \text{for}\ k = 1,\\
	F^{\pi_{k-1}}_{f_{\phi_k}} \in \tilde A_{k-1} & \text{for}\ 2 \leq k \leq l +1.
\end{cases} 
\end{align*}
Then 
\begin{enumerate}[(a)] 
\item \label{f-phi-k-value} $\delta(f_{\phi_k}) = \omega_{k}$. 
\item \label{F-1-agreement} $F_1 = F_{\phi_1} = y_1$. 
\item \label{F-k+1-agreement} For $k \geq 1$, $F_{k+1}$ is precisely the sum of all monomial terms $T$ (in $x,y_1, \ldots, y_k$) of $F_{\phi_{k+1}}$ such that $\omega(T) > \omega_{k+1}$.
\end{enumerate}
\end{lemma} 

\begin{proof}
We compute $\delta(f_{\phi_k})$ using identity \eqref{phi-delta-defn}. Let $\tilde p_k := p_1 \cdots p_{k-1}$. It is straightforward to see that $\phi_k$ has precisely $\tilde p_k$ conjugates $\phi_{k,1}, \ldots, \phi_{k,\tilde p_k}$, and $\deg_x(\tilde \phi_\delta(x,\xi)-\phi_{k,j}(x))$ is $r_1$ for $(p_1-1)p_2 \cdots p_{k-1}$ of the $\phi_{k,j}$'s, $r_2$ for $(p_2-1)p_3 \cdots p_{k-1}$ of $\phi_{k,j}$'s, and so on. Identity \eqref{phi-delta-defn} then implies that  
\begin{align*}
\delta(f_{\phi_k}) 	
	&= \delta(x)\sum_{j=1}^{\tilde p_k}\deg_x(\tilde \phi_\delta(x,\xi)-\phi_{k,j}(x)) \\
	&= p_1 \cdots p_{l+1}\left((p_1-1)p_2 \cdots p_{k-1}\frac{q_1}{p_1} + (p_2-1)p_3 \cdots p_{k-1}\frac{q_2}{p_1p_2} + \cdots + (p_{k-1} - 1)\frac{q_{k-1}}{p_1 \cdots p_{k-1}} + \frac{q_k}{p_1 \cdots p_k}\right)
\end{align*}
A straightforward induction on $k$ then yields that $\delta(f_{\phi_k}) = p_{k-1}\delta(f_{\phi_{k-1}})  + (q_k - q_{k-1}p_k)p_{k+1} \cdots p_{l+1}$. Identity \eqref{omega-defn} then implies that $\delta(f_{\phi_k}) = \omega_k$, which proves assertion \eqref{f-phi-k-value}. Assertion \eqref{F-1-agreement} follows immediately from the definitions. We now prove assertion \eqref{F-k+1-agreement}. Fix $k$, $1 \leq k \leq l$. Let $\tilde F$ be the sum of all monomial terms $T$ (in $x,y_1, \ldots, y_k$) of $F_{\phi_{k+1}}$ such that $\omega(T) > \omega_{k+1}$, i.e.\ $F_{\phi_{k+1}} = \tilde F + \tilde G$ for some $\tilde G \in \tilde A_k$ with $\omega(\tilde G) \leq \omega_{k+1}$. It follows that
\begin{align*}
\delta(\pi(\tilde F)) 
	&= \delta(\pi(F_{\phi_{k+1}}) - \pi(\tilde G)) 
	 \leq \max\{\delta(\pi(F_{\phi_{k+1}})), \delta(\pi(\tilde G))\} 
	 \leq \max\{\delta(f_{\phi_{k+1}}), \omega(\tilde G)\}
	 \leq \omega_{k+1}
\end{align*}
On the other hand, $\delta(\pi(F_{k+1})) = \delta(f_{k+1}) = \omega_{k+1}$. It follows that \begin{align}
\delta(\pi(\tilde F - F_{k+1})) 
	= \delta(\pi(\tilde F) - \pi(F_{k+1})) 
	\leq \max\{\delta(\pi(\tilde F)),  \delta(\pi(F_{k+1}))\} 
	= \omega_{k+1} \label{delta-omega}
\end{align}
Now, \eqref{F-k} and defining properties of $F_{\phi_k}$ from Lemma \ref{F-Phi} imply that $H := \tilde F - F_{k+1}$ satisfies the hypothesis of assertion \eqref{assertion-H} of Lemma \ref{delta-omega-lemma}, so that $\delta(\pi(\tilde F - F_{k+1})) = \omega(\tilde F - F_{k+1})$. Inequailty \eqref{delta-omega} then implies that 
\begin{align}
\omega(\tilde F - F_{k+1}) \leq \omega_{k+1} \label{omega-omega}
\end{align} 
But the construction of $\tilde F$ and assertion \eqref{decreasing-omega} of Lemma \ref{f-k-lemma} imply that all the monomials that appear in $\tilde F$ or $F_{k+1}$ have $\omega$-value {\em greater than} $\omega_{k+1}$. Therefore \eqref{omega-omega} implies that $\tilde F = F_{k+1}$, as required to complete the proof.
\end{proof}
%
%
%
%

The proof of the next lemma is long, and we put it in Appendix \ref{comparisondix}.

\begin{lemma}  \label{comparison-lemma-2}
Fix $k$, $0 \leq k \leq l$. Pick $\psi \in \dpsxc$ such that $\psi \equiv_{r_{k+1}} \phi_\delta$; in particular, the first $k$ Puiseux pairs of $\psi$ are $(q_1, p_1), \ldots, (q_k, p_k)$. As in \eqref{f-phi-k}, define
\begin{align*}
f_\psi^{(k)} := \prod_{j=0}^{p_0p_1 \cdots p_k-1} \left(y - \zeta^j \star_{q} \psi(x)\right), 
\end{align*}
where $q$ is the polydromy order of $\psi$ and $\zeta$ is a primitive $q$-th root of unity. Define
\begin{align}
F_{\psi}^{(k)} := \begin{cases}
	F^{\pi_1}_{f_{\psi}^{(0)}} \in \tilde A_1 & \text{for}\ k = 0,\\
	F^{\pi_k}_{f_{\psi}^{(k)}} \in \tilde A_k & \text{for}\ 1 \leq k \leq l.
\end{cases}  \label{F-psi-k-defn}
\end{align}
Then 
$$\omega\left(F_{\psi}^{(k)} - F_{k+1}\right) \leq \omega_{k+1}.$$
\end{lemma}

\subsection{Implications of polynomial key forms} \label{polynomial-section}
We continue with the notations of Section \ref{lifting-section}. Let $\xi_1, \ldots, \xi_{l+1}$ be new indeterminates, and for each $k$, $1 \leq k \leq l+1$, let $\delta_k$ be the semidegree on $\cc[x,y]$ corresponding to the generic degree-wise Puisuex series
\begin{align*} 
\tilde \phi_k(x,\xi_k) 
			&:= \phi_k(x)  +  \xi_k x^{r_k} 
\end{align*} 
i.e.\ $\delta_k(x) = p_1 \cdots p_k$ and for each $f \in \cc[x,y]\setminus \{0\}$, 
\begin{align}
\delta_k(f(x,y)) = \delta_k(x) \deg_x\left(f(x,\tilde \phi_k(x,\xi_k)\right). \label{delta-k}
\end{align} 
The following lemma follows from a straightforward examination of Algorithm \ref{key-algorithm}. 

\begin{lemma} \label{delta-k-lemma}
For each $k$, $1 \leq k \leq l+1$, the following hold:
\begin{enumerate}
\item \label{delta-k-key-forms} The key forms of $\delta_k$ are $g_0, g_1, \ldots, g_{j_k}$. 
\item The essential key forms of $\delta_k$ are $f_0, \ldots, f_k$. 
\item \label{delta-k-values} $\delta_k(g_j) = \frac{\delta(g_j)}{p_{k+1} \cdots p_{l+1}}$, $0 \leq j \leq j_k$. \qed
%
\end{enumerate} 
\end{lemma}


Fix $k$, $1 \leq k \leq l+1$. In this subsection we assume condition \eqref{polynomial-cond} below is satisfied, and examine some of its implications.
\begin{align}
\tag{$\text{Polynomial}_k$}
\parbox{.5\textwidth}
{All the key forms of $\delta_k$ are polynomials.}  \label{polynomial-cond}
\end{align}  

\begin{lemma} \label{polytivity-lemma}
Assume \eqref{polynomial-cond} holds. Then $\delta(g_j) \geq 0$ for $0 \leq j \leq j_k-1$.
\end{lemma}

\begin{proof} 
This follows from combining assertion \eqref{key-semigroup} of Proposition \ref{key-properties} with assertion \eqref{delta-k-values} of Lemma \ref{delta-k-lemma}.
\end{proof}

Let $C_k := \cc[x,y_1, \ldots, y_k] \subseteq A_k$. Since $g_j$'s are polynomial for $0 \leq j \leq j_k$, Algorithm \ref{key-algorithm} implies that $F_j$'s (defined in \eqref{F-k}) are also polynomial for $0 \leq j \leq k$; in particular, $F_1 \in C_1$ and $F_{j+1} \in C_{j}$, $1 \leq j \leq k-1$. For $1 \leq j \leq k-1$, let $H_{j+1}$ be the {\em leading form} of $F_{j+1}$ with respect to $\omega$, i.e.\
\begin{align}
H_{j+1} := y_{j}^{p_j} - c_{j,0} x^{\beta^0_{j,0}} y_1^{\beta^0_{j,1}} \cdots y_{j-1}^{\beta^0_{j,j-1}},\quad 1 \leq j \leq k-1. \label{H-defn}
\end{align} 
Let $\prec$ be the reverse lexicographic order on $C_k$, i.e.\ $x^{\beta_0}y_1^{\beta_1} \cdots y_k^{\beta_k} \prec x^{\beta'_0}y_1^{\beta'_1} \cdots y_k^{\beta'_k}$ iff the right-most non-zero entry of $(\beta_0 - \beta'_0, \ldots, \beta_k - \beta'_k)$ is negative.\\

The following lemma is the main result of this subsection. Its proof is long, and we put it in Appendix \ref{polynomial-appendix}

\begin{lemma} \label{poly-positive-prop}
Assume \polsub{l+1} holds. Then 
\begin{enumerate}
\item \label{wp-assertion} Recall the notation of Section \ref{degree-like-section}, and define 
$$S^\delta := \dsum_{d \in \zz} \scrF^\delta_d \supseteq \cc[x,y]^{\delta}$$
Then $S^\delta$ is generated as a $\cc$-algebra by $(1)_1, (x)_{\omega_0}, (y_1)_{\omega_1}, \ldots,  (y_{l+1})_{\omega_{l+1}}$. 
\item \label{grobner-assertion} Let $J_{l+1}$ be the ideal in $C_{l+1}$ generated by the {\em leading weighted homogeneous forms} (with respect to $\omega$) of polynomials $F \in C_{l+1}$ such that $\delta(\pi(F)) < \omega(F)$. Then $\scrB_{l+1} :=(H_{l+1}, \ldots, H_2)$ is a Gr\"obner basis of $J_{l+1}$ with respect to $\prec$.
\end{enumerate}
\end{lemma}


\section{Proof of the main results} \label{proof-section}
In this section we give proofs of Theorems \ref{effective-answer}, \ref{geometric-answer}, \ref{graph-thm} and \ref{cousin-answer}. 

%

\begin{proof}[\bf Proof of Theorem \ref{cousin-answer}]

The implication \eqref{all-cousinomials} $\im$ \eqref{last-cousinomial} is obvious. It is straightforward to see the following reformulation of assertion \eqref{cousin-polynomial}:

\begin{lemma} \label{cousin-polynomial-lemma}
Assertion \eqref{cousin-polynomial} equivalent to the following assertion:
\begin{enumerate}
\let\oldenumi\theenumi
\renewcommand{\theenumi}{$\oldenumi'$}
\item \label{cousin-polynomial'} There exists a polynomial $f \in \cc[x,y]$ such that for each analytic branch $C$ of the curve $f = 0$ at infinity, 
\begin{itemize}
	\item $C$ intersects $L_\infty$ at $O$,
	\item  $C$ has a {\em descending} Puiseux expansion $y = \theta(x)$ at $O$ such that $\deg_x(\theta - \phi) \leq 1-r$. \qed
\end{itemize}
\end{enumerate}
\end{lemma} 
Assertion \eqref{last-expansion} of Theorem \ref{key-thm} implies that if $g_{n+1}$ is a polynomial, then $g_{n+1}$ satisfies the properties of $f$ from \eqref{cousin-polynomial'}; in particular \eqref{last-cousinomial} $\im$ \eqref{cousin-polynomial'}. To finish the proof of Theorem \ref{cousin-answer} it remains to prove that \eqref{cousin-polynomial'} $\im$ \eqref{all-cousinomials}. So assume \eqref{cousin-polynomial'} holds. We proceed by contradiction, i.e.\ we also assume that there exists $m$, $1 \leq m \leq n$, such that $g_{m+1}$ is not a polynomial, and show that this leads to a contradiction. By assertion \eqref{truncassertion} of Proposition \ref{key-properties}, we may (and will) assume that $m = n$.\\

We adopt the notations of Sections \ref{lifting-section} and \ref{polynomial-section}. In particular, we write $\tilde \phi_\delta(x,\xi)$ (resp.\ $\phi_\delta(x)$, $r_\delta$) for $\tilde \phi(x,\xi)$ (resp.\ $[\phi(x)]_{> 1-r}$, $1-r$), and denote by $(q_1, p_1), \ldots, (q_{l + 1}, p_{l+1})$ the formal Puiseux pairs of $\tilde \phi_\delta$. We also denote by $g_{j_0}, \ldots, g_{j_{l+1}}$ the sequence of essential key forms, and set $f_k := g_{j_k}$, $0 \leq k \leq l+1$.\\

Let $f \in \cc[x,y]$ be as in \eqref{cousin-polynomial'}. By assumption $f$ has a descending Puiseux factorization of the form
\begin{align}
f &= a\prod_{m=1}^M f_{\psi_m} \label{f-decomposition} 
\end{align}
for some $a \in \cc^*$ and $\psi_1, \ldots, \psi_m \in \dpsxc$ such that 
\begin{align}
\psi_m &\equiv_{r_\delta} \phi_\delta,\quad \text{for each}\ m,\ 1 \leq m \leq M, \label{psi-n-equivalence}
\end{align}
where $f_{\psi_m}$'s are defined as in \eqref{f-phi-eqn}. W.l.o.g.\ we may (and will) assume that $a = 1$. \\

At first we claim that $l \geq 1$. Indeed, assume to the contrary that $l=0$. Then
\begin{align*}
\tilde \phi_\delta(x,\xi) = h(x) + \xi x^{r_\delta}
\end{align*}
for some $h \in \cc[x,x^{-1}]$. Since $g_{n+1}$ is not a polynomial, assertion \eqref{key-semigroup} of Proposition \ref{key-properties} implies that $h(x) = h_1(x)+ h_2(x)$ where $h_1 \in \cc[x]$, $h_2 \in \cc[x^{-1}]\setminus \cc$, and $0 > \deg_x(h_2(x)) > r_\delta$. Let $e := -\deg_x(h_2(x)) > 0$ and $y' := y - h_1(x)$. Then \eqref{f-decomposition} implies that $f$ is a product of elements in $\dpsxc[ y']$ of the form $ y' - \psi_{m,i}(x)$ for $\psi_{m,i} \in \dpsxc$ such that each $\psi_{m,i}(x) = h_2(x) + \ldt$, where $\ldt$ denotes terms with degree smaller than $\ord_x(h_2) < -e$. It is then straightforward to see that $f \not\in \cc[x, y'] = \cc[x,y]$, which contradicts our choice of $f$. It follows that $l \geq 1$, as claimed.\\

Let $F_k$, $1 \leq k \leq l+1$, be as in \eqref{F-k}. Fix $m$, $1 \leq m \leq M$. Then \eqref{psi-n-equivalence} and Lemma \ref{comparison-lemma-2} imply that 
\begin{align}
F_{\psi_m}^{(l)} = F_{l+1} + \tilde F_m
\end{align}
where $\tilde F_m \in \tilde A_l := \dpsxc[y_1, \ldots, y_l]$ and $\omega(\tilde F_m) \leq \omega_{l+1}$. Let $s_m$ denote the polydromy order of $\psi_m$ and $\mu_m$ be a primitive $s_m$-th root of unity. Identity \eqref{psi-n-equivalence} implies that $t_m := s_m/(p_1p_2 \cdots p_{l})$ is a positive integer. Therefore \eqref{f-phi-n-m} and assertion \eqref{star-projection} of Lemma \ref{star-lemma} imply that 
\begin{align}
f_{\psi_m}  
&= \prod_{j=0}^{t_m-1} \mu_k^{jp_1 \cdots p_l} \star_{s_m} \left(f_{\psi_m}^{(l)}\right)
= \prod_{j=0}^{t_m-1} \mu_m^{jp_1 \cdots p_l} \star_{s_m} \left(\pi_l \left(F_{l+1} + \tilde F_m\right) \right) 
= \pi_l(G_m)\quad \text{where} \label{f-psi-n} \\
G_m &:= \prod_{j=0}^{t_m-1} \left(F_{l+1} + \mu_m^{jp_1 \cdots p_l} \star_{s_m}\left(\tilde F_m \right) \right) \in \tilde B_l. \label{G-n-defn}
\end{align}
Recall that $F_{l+1} = y_{l}^{p_{l}} - \sum_{i=1}^{m_l}c_{l,i} x^{\beta^i_{l,0}} y_1^{\beta^i_{l,1}} \cdots y_{l}^{\beta^i_{l,l}}$. Since by our assumption all the key forms but the last one are polynomials, it follows from assertion \eqref{key-semigroup} of Proposition \ref{key-properties} that $\beta^{i}_{l,0} \geq 0$ for all $i < m_l$, but $\beta^{i}_{m_l,0} < 0$; set 
\begin{align}
\omega'_{l+1} := \omega\left(x^{\beta^{m_l}_{l,0}} y_1^{\beta^{m_l}_{l,1}} \cdots y_{l}^{\beta^{m_l}_{l,l}}\right) = \sum_{i=0}^l \beta^{m_l}_{l,i}\omega_i.
\end{align}
Then $\omega'_{l+1} > \omega_{l+1}$ and therefore we may express $G_m$ as
\begin{align}
G_m	&= \prod_{j=0}^{t_m - 1} \left(y_{l}^{p_{l}} - \sum_{i=0}^{m_l}c_{l,i} x^{\beta^i_{l,0}} y_1^{\beta^i_{l,1}} \cdots y_{l}^{\beta^i_{l,l}} - G_{m,j} \right) , \label{G-n}
\end{align} 
for some $G_{m,j} \in \tilde B_l$ with $\omega(G_{m, j}) < \omega'_{l+1}$. Identities \eqref{f-decomposition}, \eqref{f-psi-n} and \eqref{G-n} imply that $f = \pi_l(F)$ for some $F \in \tilde A_l$ of the form
\begin{align}
F = \prod_{m'=1}^{M'} \left(y_{l}^{p_{l}}  - \sum_{i=0}^{m_l}c_{l,i} x^{\beta^i_{l,0}} y_1^{\beta^i_{l,1}} \cdots y_{l}^{\beta^i_{l,l}} - G'_{m'} \right),
\end{align}
where $\omega(G'_{m'}) < \omega'_{l+1}$ for all $m'$, $1 \leq m' \leq M'$. Let 
\begin{align*}
G := \begin{cases}
F - y_l^{M'p_{l}}	
& \text{if}\ m_l=0,\\
F - \left(y_l^{p_{l}} - \sum_{i=0}^{m_l-1}c_{l,i} x^{\beta^i_{l,0}} y_1^{\beta^i_{l,1}} \cdots y_{l}^{\beta^i_{l,l}} \right)^{M'}
& \text{otherwise.}
\end{cases}
\end{align*}
Recall from assertion \eqref{zero} of Lemma \ref{f-k-lemma} that $\beta^0_{l,l} = 0$. It follows that the {\em leading weighted homogeneous form} of $ G$ with respect to $\omega$ is
\begin{align} \label{ld-omega}
\ld_{\omega}(G) = 
\begin{cases}
- c_{l,0}x^{\beta^0_{l,0}} y_1^{\beta^0_{l,1}} \cdots y_{l-1}^{\beta^0_{l,l-1}} & \text{if}\ m_l = 0\ \text{and}\ M' = 1, \\
\sum_{i=1}^{M'} \binom{M'}{i} (-c_{l,0})^i y_l^{(M'-i)p_{l}} x^{i\beta^0_{l,0}} y_1^{i\beta^0_{l,1}} \cdots y_{l-1}^{i\beta^0_{l,l-1}} & \text{if}\ m_l = 0, \text{and}\ M' > 1,\\
M'c_{l,m_l}\left(y_l^{p_{l}} - c_{l,0}x^{\beta^0_{l,0}} y_1^{\beta^0_{l,1}} \cdots y_{l-1}^{\beta^0_{l,l-1}}\right)^{M'-1} x^{\beta^{m_l}_{l,0}} y_1^{\beta^{m_l}_{l,1}} \cdots y_{l}^{\beta^{m_l}_{l,l}} & \text{otherwise.}
\end{cases}
\end{align}
Since $\pi_l(F) = f \in \cc[x,y]$, it follows that $g:= \pi_l(G)$ is also a polynomial in $x$ and $y$. Assertion \eqref{wp-assertion} of Lemma \ref{poly-positive-prop} then implies that there is a polynomial $\tilde G \in C_l := \cc[x,y_1, \ldots, y_l]$ such that $\pi_l(\tilde G) = g$ and $\omega(\tilde G) = \delta_l(g)$. In particular, $\omega(\tilde G) \leq \omega(G)$.

\begin{claim} \label{omega-and-tilde}
$\omega(\tilde G) = \omega(G)$.
\end{claim}

\begin{proof}
Let $\prec$ be the reverse lexicographic monomial ordering on $C_l$ from Section \ref{polynomial-section} and let $\alpha$ be the smallest positive integer such that $x^\alpha \ld_{\omega}(G)$ is a polynomial. Then \eqref{ld-omega} implies that the {\em leading term} of $x^\alpha \ld_{\omega}(G)$ with respect to $\prec$ is

\begin{align} \label{lt-G}
\lt_{\prec}\left(x^\alpha\ld_{\omega}\left(G \right)\right) 
= \begin{cases}
- c_{l,0} y_1^{\beta^0_{l,1}} \cdots y_{l-1}^{\beta^0_{l,l-1}} 
& \text{if}\ m_l = 0\ \text{and}\ M' = 1, \\
- c_{l,0} M' y_l^{(M'-1)p_{l}} x^{(M'-1)|\beta^0_{l,0}|} y_1^{\beta^0_{l,1}} \cdots y_{l-1}^{\beta^0_{l,l-1}} 
& \text{if}\ m_l = 0, \text{and}\ M' > 1,\\
M'c_{l,m_l}y_l^{(M'-1)p_{l}+ \beta^{m_l}_{l,l}} y_1^{\beta^{m_l}_{l,1}} \cdots y_{l-1}^{\beta^{m_l}_{l,l-1}} & \text{otherwise.}
\end{cases}
\end{align}
Assume contrary to the claim that $\omega(G) > \omega(\tilde G) = \delta_l(g)$. Then $x^\alpha \ld_\omega(G) \in J_l$, where $J_l$ is the ideal from assertion \eqref{grobner-assertion} of Lemma \ref{poly-positive-prop}. Assertion \eqref{grobner-assertion} of Lemma \ref{poly-positive-prop} then implies that there exists $j$, $1 \leq j \leq l-1$, such that $y_{j}^{p_{j}} = \lt_{\prec}(H_{j+1})$ divides $\lt_{\prec}(x^\alpha\ld_{\omega}(G))$. But this contradicts the fact that $\beta^{m_l}_{l,j'} < p_{j'}$ for all $j'$, $1 \leq j' \leq l-1$ (assertion \eqref{exponent-bound} of Lemma \ref{f-k-lemma}) and completes the proof of the claim.
\end{proof}

Let $J_l$ and $\alpha$ be as in the proof of Claim \ref{omega-and-tilde}. Note that $\ld_{\omega}(x^\alpha \tilde G) \not\in J_l$ by our choice of $\tilde G$. Therefore, after `dividing out' $\tilde G$ by the Gr\"obner basis $\scrB_l$ of Lemma \ref{poly-positive-prop} (which does not change $\omega(\tilde G)$) if necessary, we may (and will) assume that 
\begin{align}
\parbox{.9\textwidth}{$y_{j}^{p_{j}}$ does {\em not} divide any of the monomial terms of $\ld_\omega(x^\alpha \tilde G)$ for any $j$, $1 \leq j \leq l-1$.} \label{tilde-G-division}
\end{align}
Since $\pi_l(x^\alpha G - x^\alpha \tilde G) = 0$, it follows that $\ld_\omega(x^\alpha G - x^\alpha \tilde G) \in J_l$. Since $\omega(G) = \omega(\tilde G)$ by Claim \ref{omega-and-tilde}, it follows that $H^* := \ld_\omega(x^\alpha G) - \ld_\omega(x^\alpha \tilde G) \in J_l$. Let 
\begin{align*}
H := \lt_{\prec}\left(\ld_\omega(x^\alpha G)\right)\ \text{and}\ \tilde H := \lt_{\prec}\left(\ld_\omega(x^\alpha \tilde G)\right).
\end{align*}
Since $\tilde G \in \cc[x,y_1, \ldots, y_l]$, it follows that $\deg_x(\tilde H) \geq \alpha$. On the other hand, \eqref{lt-G} implies that $\deg_x(H) = \alpha + \beta^0_{m_l,0} < \alpha$. It follows in particular that $H \neq \tilde H$ and $\lt_{\prec}(H^*) = \max_{\prec}\{H, -\tilde H\}$. Then \eqref{lt-G} and \eqref{tilde-G-division} imply that $y_{j}^{p_{j}} = \lt_{\prec}(H_{j})$ does {\em not} divide $\lt_{\prec}(H^*)$ for any $j$, $1 \leq j \leq l-1$. This contradicts assertion \eqref{grobner-assertion} of Lemma \ref{poly-positive-prop} and finishes the proof of the implication \eqref{cousin-polynomial'} $\im$ \eqref{all-cousinomials}, as required to complete the proof of Theorem \ref{cousin-answer}.
\end{proof}

\begin{proof}[\bf Proof of Theorem \ref{effective-answer}]
Theorem \eqref{cousin-answer} implies that \eqref{all-polynomials} $\Leftrightarrow$ \eqref{last-polynomial}. Now assume \eqref{all-polynomials} is true. Note that $\delta(f) > 0$ for each non-constant $f \in \cc[x,y]$ (since such an $f$ must have a pole at the irreducible curve $E'_1 := \pi'(E_1)\subseteq Y'$); so that the ring $S^\delta$ defined in Lemma \ref{poly-positive-prop} is precisely the ring $\cc[x,y]^\delta$ from Section \ref{degree-like-section}. Assertion \eqref{wp-assertion} of Lemma \ref{poly-positive-prop} and Proposition \ref{basic-semi-prop} then imply that $Y'$ is isomorphic to the closure of the image of $\cc^2$ in the weighted projective variety $\pp^{l+2}(1,\delta(f_0), \ldots, \delta(f_{l+1}))$ under the mapping $(x,y) \mapsto [1:f_0:\cdots:f_{l+1}]$. This in particular shows that \eqref{all-polynomials} $\im$ \eqref{algebraic-Y'}.\\

It remains to show that \eqref{algebraic-Y'} $\im$ \eqref{all-polynomials}. So assume that $Y'$ is algebraic. Recall the set up of Proposition \ref{prop-param}. We can identify $Y'$ with $\bar X$ and $E'_1$ with $C_\infty$ (where $\bar X$ and $C_\infty$ are as in Proposition \ref{prop-param}). Let $P_\infty \in C_\infty$ be as in Proposition \ref{prop-param}. Since $Y'$ is algebraic, there exists a compact algebraic curve $C$ on $Y'$ such that $P_\infty \not\in C$. Let $f \in \cc[x,y]$ be the polynomial that generates the ideal of $C$ in $\cc[x,y]$. Proposition \ref{prop-param} then implies that $f$ satisfies the condition of property \eqref{cousin-polynomial'} from Lemma \ref{cousin-polynomial-lemma}. Theorem \ref{cousin-answer} and Lemma \ref{cousin-polynomial-lemma} then show that \eqref{all-polynomials} is true, as required. 
\end{proof}

\begin{proof}[\bf Proof of Theorem \ref{geometric-answer}]
Let $\delta$ be the semidegree on $\cc[x,y]$ corresponding to the curve at infinity on $\bar X$, $\tilde \phi_\delta(x,\xi)$ be the associated generic descending Puiseux series, and $g_0, \ldots, g_{n+1} \in \cc[x,x^{-1},y]$ be the key forms. If $\bar X$ is algebraic, then Theorem \ref{effective-answer} implies that $g_{n+1}$ is a polynomial. Proposition \ref{prop-param} and assertion \eqref{last-expansion} of Theorem \ref{key-thm} then imply that $g_{n+1} = 0$ defines a curve $C$ as in assertion \eqref{one-place-curve} of Theorem \ref{geometric-answer}. This proves the implication \eqref{algebraic-X-bar} $\im$ \eqref{one-place-curve}, and also the last assertion of Theorem \ref{geometric-answer}. It remains to prove the implication \eqref{one-place-curve} $\im$ \eqref{algebraic-X-bar}. So assume there exists $f \in \cc[x,y]$ such that $C := \{f = 0\}$ is as in \eqref{one-place-curve}. Proposition \ref{prop-param} implies that $f$ satisfies the condition of property \eqref{cousin-polynomial'} from Lemma \ref{cousin-polynomial-lemma}. Lemma \ref{cousin-polynomial-lemma}, Theorem \ref{cousin-answer} and Theorem \ref{effective-answer} then imply that $\bar X$ is algebraic, as required. 
\end{proof}

\begin{proof}[\bf Proof of Theorem \ref{graph-thm}]
%

The $(\im)$ direction of assertion \eqref{algebraic-graph} follows from Theorem \ref{effective-answer} and assertion \eqref{key-semigroup} of Proposition \ref{key-properties}. For the $(\Leftarrow)$ implication, note that assertion \eqref{alpha=p} of Proposition \ref{essential-value-prop} and Property \eqref{semigroup-property} of key forms imply that for each $k$, $1 \leq k \leq l$,
\begin{align*}
p_k\omega_k = \sum_{j = 0}^{k-1}\beta'_{k,j}\omega_j
\end{align*}
where $\beta'_{k,j}$'s are integers such that $0 \leq \beta'_{k,j} < p_j$ for $1 \leq j <k$. Define $g^0_k$, $0  \leq k \leq l+1$, as follows:
\begin{align*}
g^0_k &= 
			\begin{cases}
				x &\text{if}\ k = 0,\\
				y &\text{if}\ k = 1,\\
				(g^0_{k-1})^{p_{k-1}} - \prod_{j=0}^{k-2} (g^0_j)^{\beta'_{k-1,j}} &\text{if}\ 2 \leq k \leq l+1.
			\end{cases}
\end{align*}
Assertion \eqref{key-to-degree} of Theorem \ref{key-thm} implies that there is a unique semidegree $\delta^0$ on $\cc[x,y]$ such that its key forms are $g^0_0, \ldots, g^0_{l+1}$ and $\delta^0(g^0_k) =  \omega_k$, $0 \leq k \leq l+1$. Since $\omega_{l+1} >  0$ (assertion \ref{converse-graph} of Theorem \ref{primitive-resolution-graph}) it follows that $\delta^0$ defines a primitive normal compactification $\bar X^0$ of $\cc^2$ (Remark \ref{analytic-contractibility}). It follows from Proposition \ref{essential-value-prop} that $(q_k,p_k)$, $1 \leq k \leq l+1$, are uniquely determined in terms of $\omega_0, \ldots, \omega_{l+1}$. Therefore $\Gamma$ is precisely the augmented and marked dual graph associated to the minimal $\pp^2$-dominating resolution of singularities of $\bar X^0$. Now, if \eqref{semigroup-criterion-1} holds for each $k$, $1 \leq k \leq l$, then each $\beta'_{k,0}$ is non-negative, so that each $g^0_k$ is a polynomial. Theorem \ref{effective-answer} then implies that $\bar X^0$ is algebraic, which proves the $(\Leftarrow)$ implication of assertion \eqref{algebraic-graph}. \\

Now we prove assertion \eqref{non-algebraic-graph}. For the $(\im)$ implication, pick a non-algebraic normal primitive compactification $\bar X$ of $\cc^2$ such that $\Gamma = \Gamma_{\bar X}$. Let $\delta$ be the order of pole along the curve at infinity on $\bar X$. Theorem \ref{effective-answer} implies that at least one of the key forms of $\delta$ is not a polynomial. Assertions \eqref{omega-from-p} and \eqref{essential-group} of Proposition \ref{essential-value-prop} and assertion \eqref{key-semigroup} of Proposition \ref{key-properties} now imply that either \eqref{semigroup-criterion-1} or \eqref{semigroup-criterion-2} fails, as required. It remains to prove the $(\Leftarrow)$ implication of assertion \eqref{non-algebraic-graph}. Let $g^0_k$, $0 \leq k \leq l+1$, be as in the preceding paragraph. If \eqref{semigroup-criterion-1} fails for some $k$, $1 \leq k \leq l$, take the smallest such $k$. Then by construction $g^0_k$ is {\em not} a polynomial, so that $\bar X^0$ is not algebraic (Theorem \ref{effective-answer}), as required. Now assume that \eqref{semigroup-criterion-1} holds for all $k$, $1 \leq k \leq l$, but there exists $k$, $1 \leq k \leq l$, such that \eqref{semigroup-criterion-2} fails; let $k$ be the smallest integer such that \eqref{semigroup-criterion-2} fails. Pick $\tilde \omega \in (\omega_{k+1}, p_k \omega_k) \cap \zz\langle \omega_0, \ldots, \omega_k \rangle \setminus \zz_{\geq 0}\langle \omega_0, \ldots, \omega_k \rangle$. Then it is straightforward to see that there exist integers $\tilde \beta_0, \ldots, \tilde \beta_k $ such that $\tilde \beta_0 < 0$, $0 \leq \tilde \beta_j < p_j$, $1 \leq j <k$, and 
 \begin{align*}
\tilde \omega = \sum_{j = 0}^{k-1}\tilde \beta_j \omega_j
\end{align*}
Define $g^1_i$, $0  \leq i \leq l+2$, as follows:
\begin{align*}
g^1_i &= 
			\begin{cases}
				g^0_i &\text{if}\ 0 \leq i \leq k+1,\\
				g^0_{k+1}- \prod_{j=0}^{k} (g^0_j)^{\tilde \beta_j} &\text{if}\ i = k + 2,\\
				(g^1_{i-1})^{p_{i-2}} - \prod_{j=0}^{k} (g^1_j)^{\beta'_{i-2,j}}\prod_{j=k+2}^{i-2} (g^1_j)^{\beta'_{i-2,j-1}} &\text{if}\ k+3 \leq i \leq l+2.
			\end{cases}
\end{align*}
The same arguments as in the proof of assertion \eqref{algebraic-graph} imply that there is a primitive normal compactification $\bar X^1$ of $\cc^2$ such that 
\begin{itemize}
\item the semidegree $\delta^1$ corresponding to its curve at infinity  has key forms $g^1_0, \ldots, g^1_{l+2}$, and 
$$
\delta^1(g^1_i) = 
			\begin{cases}
				\omega_i &\text{if}\ 0 \leq i \leq k,\\
				\tilde \omega &\text{if}\ i = k+1,\\
				\omega_{i-1} &\text{if}\ k+2 \leq i \leq l+2.
			\end{cases}
$$
\item $\Gamma$ is the augmented and marked dual graph associated to the minimal $\pp^2$-dominating resolution of singularities of $\bar X^1$. 
\end{itemize}
Since $g^1_{k+2}$ is not a polynomial, $\bar X^1$ is not algebraic (Theorem \ref{effective-answer}), as required to complete the proof of assertion \eqref{non-algebraic-graph}.
\end{proof}

\appendix 

\section{Proof of Lemma \ref{comparison-lemma-2}} \label{comparisondix}

\begin{notation}
Fix $k$, $1 \leq k \leq l+1$. For $F \in \tilde A_k$ and $\mu \in \rr$, we write $[F]_{> \mu}$ for the sum of all monomial terms $H$ of $F$ such that $\omega(H) > \mu$.
\end{notation}

\begin{lemma} \label{F-phi-psi-inductive-agreement}
Fix $k$, $1 \leq k \leq l$. Pick $\psi_1, \psi_2 \in \dpsxc$ and $\mu \leq \omega_k \in \rr$. Assume 
\begin{compactenum}
\item \label{Puiseux-assumption} the first $k$ Puiseux pairs of each $\psi_j$ are $(q_1, p_1), \ldots, (q_k,p_k)$. 
\end{compactenum}
Assumption \eqref{Puiseux-assumption} implies that we can define $F_{\psi_j}^{(k-1)}, F_{\psi_j}^{(k)}$, $1 \leq j \leq 2$, as in Lemma \ref{comparison-lemma-2}. Assume 
\begin{compactenum}
\setcounter{enumi}{1}
\item \label{mu-assumption} $[F_{\psi_1}^{(k-1)}]_{> \mu}  = [F_{\psi_2}^{(k-1)}]_{> \mu}$, and
\item \label{omega-k-assumption} $[F_{\psi_j}^{(k-1)}]_{> \omega_k} = [F_k]_{> \omega_k}$ for each $j$, $1 \leq j \leq 2$.
\end{compactenum}
Then $[F_{\psi_1}^{(k)}]_{>(p_{k} -1)\omega_{k} + \mu} = [F_{\psi_2}^{(k)}]_{>(p_{k} -1)\omega_{k} + \mu}$.
\end{lemma}

\begin{proof}
Let 
\begin{align*}
\tilde A:= \begin{cases}
\tilde A_1 & \text{if}\ k = 1, \\
\tilde A_{k-1} & \text{otherwise.}
\end{cases}
\end{align*}
Assumptions \ref{mu-assumption} and \ref{omega-k-assumption} imply that there exists $G \in \tilde A$ with $\omega(G) \leq \omega_k$ such that for both $j$, $1 \leq j \leq 2$,
$$F_{\psi_j}^{(k-1)} = F_k + G + G_j$$
for some $G_j \in \tilde A$ with $\omega(G_j) \leq \mu$. Fix $j$, $1 \leq j \leq 2$. Let $m_j$ be the polydromy order of $\psi_j$ and $\mu_j$ be a primitive $m_j$-th root of unity. Then identity \eqref{f-phi-n-m} and assertion \eqref{star-projection} of Lemma \ref{star-lemma} imply that 
\begin{align*}
f_{\psi_j}^{(k)} &= \prod_{i=0}^{p_k-1} \mu_j^{ip_1 \cdots p_{k-1}} \star_{m_j} \left(f_{\psi_j}^{(k-1)}\right) = \pi_{k-1}(G^*_j)\quad\ \text{where}\\
G^*_j 	&:= \prod_{i=0}^{p_k-1} \mu_j^{ip_1 \cdots p_{k-1}} \star_{m_j} \left(F_{\psi_j}^{(k-1)}\right) 
= \prod_{i=0}^{p_k-1} \mu_j^{ip_1 \cdots p_{k-1}} \star_{m_j} \left(F_k + G + G_j \right) \\
&= \prod_{i=0}^{p_k-1} \left(F_k + \mu_j^{ip_1 \cdots p_{k-1}} \star_{m_j}G + \mu_j^{ip_1 \cdots p_{k-1}} \star_{m_j}G_j \right)
= \prod_{i=0}^{p_k-1} \left(F_k + \mu^{ip_1 \cdots p_{k-1}} \star_{m}G + \mu_j^{ip_1 \cdots p_{k-1}} \star_{m_j}G_j \right)
\end{align*}
where $m$ is the {\em polydromy order} of $G$ (Definition \ref{star-defn}) and $\mu$ is a primitive $m$-th root of unity (the last equality is an implication of assertion \eqref{different-stars} of Lemma \ref{star-lemma}). Let 
\begin{align} \label{zeroth-step}
G_{j,0} &:= \prod_{i=0}^{p_k-1} \left(y_k + \mu^{ip_1 \cdots p_{k-1}} \star_{m}G + \mu_j^{ip_1 \cdots p_{k-1}} \star_{m_j}G_j \right) \in \tilde A_k.
\end{align}
Note that $\pi_k(G_{j,0}) = f_{\psi_j}^{(k)} = \pi_k(F_{\psi_j}^{(k)})$. Now we construct $F_{\psi_j}^{(k)}$ from $G_{j,0}$ via constructing a sequence of elements $G_{j,0}, G_{j,1}, \ldots $ as follows:
\begin{itemize}
\item  For $\beta := (\beta_1, \ldots, \beta_{k}) \in \zz_{\geq 0}^{k}$, define
$$|\beta|_{k-1} := \sum_{j=1}^{k-1} p_0 \ldots p_{j-1} \beta_j.$$
Consider the well order $\prec^*_{k-1}$ on $\zz_{\geq 0}^{k}$ defined as follows: $ \beta \prec^*_{k-1} \beta'$ iff 
\begin{enumerate}
\item $|\beta|_{k-1} <  |\beta'|_{k-1}$, or  
\item $|\beta|_{k-1} =  |\beta'|_{k-1}$ and the left-most non-zero entry of $\beta - \beta'$ is negative. 
\end{enumerate}
\item Assume $G_{j,N}$ has been constructed, $N \geq 0$. Express $G_{j,N}$ as 
\begin{align*}
G_{j,N} 
	&= \sum_{\beta \in \zz_{\geq 0}^{k}} g_{j,N,\beta}(x) y_1^{\beta_1}\cdots y_{k}^{\beta_{k}} 
\end{align*} 
and define
\begin{align*}
\scrE_{j,N} &:= \{\beta \in \zz_{\geq 0}^{k}: g_{j,N,\beta} \neq 0\ \text{and}\ \beta_i \geq p_i\ \text{for some}\ i, 1 \leq i \leq k-1 \}.
\end{align*}
\item If $\scrE_{j,N}  = \emptyset$, then {\bf stop.} 
\item Otherwise pick the maximal element $\beta^* = (\beta^*_1, \ldots, \beta^*_k)$ of $\scrE_{j,N}$ with respect to $\prec^*_{k-1}$, and the maximal $i^*$, $1 \leq i^* \leq k-1$, such that $\beta^*_{i^*} \geq p_{i^*}$, and set 
\begin{align}
G_{j,N+1} 
	&= \sum_{\beta \neq \beta^*} g_{j,N,\beta}(x) y_1^{\beta_1}\cdots y_{k}^{\beta_{k}} 
		+  g_{j,N,\beta^*}(x) \prod_{i \neq i^*} (y_i)^{\beta^*_i}(y_i)^{\beta^*_{i^*} - p_{i^*}}
			(y_{i+1} - (F_{i+1} - y_{i^*}^{p_{i^*}})) 
	\label{G-next}
\end{align}
\end{itemize}

Assertion \eqref{F-k+1-agreement} of Lemma \ref{comparison-lemma-1} and assertion \eqref{F-Phi-exponents} of Lemma \ref{F-Phi} imply that the all the `new' exponents of $(y_1, \ldots, y_k)$ that appear in $G_{j,N+1}$ are smaller (with respect to $\prec^*_{k-1}$) than $\beta^*$, and it follows that the sequence of $G_{j,N}$'s stop at some finite value $N^*$ of $N$. 

\begin{proclaim} \label{last-step}
$G_{j,N^*} = F_{\psi_j}^{(k)}$.
\end{proclaim}

\begin{proof}
Indeed, \eqref{G-next} implies that $\pi_k(G_{j,N^*}) = \pi_k(G_{j,0}) = f_{\psi_j}^{(k)}$. Since, we must have $\scrE_{j,N^*} = \emptyset$ for $G_{j,N^*}$ to be the last element of the sequence of $G_{j,N}$'s, $G_{j,N^*}$ satisfies the characterizing properties of $F_{\psi_j}^{(k)} = F^{\pi_k}_{f_{\psi_j}^{(k)}}$ from Lemma \ref{F-Phi}. 
\end{proof}

Now note that, for each $i$, $1 \leq i \leq k-1$, every monomial term in $y_{i+1} - (F_{i+1} - y_i^{p_i})$ has $\omega$-value smaller than or equal to $\omega(y_i^{p_i})$ (assertion \eqref{decreasing-omega}, Lemma \ref{f-k-lemma}). It then follows from \eqref{zeroth-step} and \eqref{G-next} that $G_j$ has no effect on $[G_{j,N}]_{> (p_k-1)\omega_k + \mu}$ for any $N$; i.e.\ $[G_{1,N}]_{> (p_k-1)\omega_k + \mu} = [G_{2,N}]_{> (p_k-1)\omega_k + \mu}$ for all $N$. Claim \ref{last-step} then implies the lemma.
\end{proof}

\begin{cor} \label{stronger-phi-psi-agreement}
Let $\omega_{i,j} := \omega_i + q_{j}p_{j+1} \cdots p_{l+1} - q_ip_{i+1} \cdots p_{l+1}$ for $1 \leq i\leq j \leq l+1$. 
Fix $j$, $0 \leq j \leq l$. Let $\psi \in \dpsxc$ be such that $\psi \equiv_{r_{j+1}} \phi_{j+1}$ (where $r_1, \ldots, r_{l+1}$ and $\phi_1, \ldots, \phi_{l+1}$ are as in \eqref{r-k}, \eqref{phi-k} respectively). Then for all $i$ such that $0 \leq i \leq j$,
$$\left[F_{\psi}^{(i)}\right]_{> \omega_{i+1,  j+1}} = \left[F_{\phi_{j+1}}^{(i)}\right]_{> \omega_{i+1,  j+1}}.$$
\end{cor}

\begin{proof}
At first we consider the $i = 0$ case. \eqref{f-phi-eqn} implies that $f_{\psi}^{(0)} = y - \psi(x)$ and $f_{\phi_{k+1}}^{(0)} = y - \phi_{k+1}(x)$. Then \eqref{F-psi-k-defn} implies that
\begin{align*}
F_{\psi}^{(0)} &= y_1 + \phi_1(x) - \psi(x) \\
F_{\phi_{k+1}}^{(0)} &= y_1 + \phi_1(x) - \phi_{j+1}(x)
\end{align*} 
It follows that 
\begin{align*}
\omega\left(F_{\psi}^{(0)} - F_{\phi_{j+1}}^{(0)}\right)
	& = \omega_0\deg_x(\phi_{j+1}(x) - \psi(x)) 
	\leq p_1 \cdots p_{l+1}r_{j+1} = q_{j+1}p_{j+2} \cdots p_{l+1} = \omega_{1,j+1}
\end{align*}
It follows that the corollary is true for $i=0$ and all $j$, $0 \leq j \leq l$. \\

Now we start the proof of the general case. We proceed by induction on $j$. It follows from the preceding discussion that the corollary is true for $j=0$. So assume it holds for $0 \leq j \leq j' \leq l-1$. To show that it holds for $j = j'+1$, we proceed by induction on $i$. By the same reasoning, we may assume that it also holds for $j=j'+1$ and $0 \leq i \leq i' \leq j'$. Pick $\psi$ such that $\psi \equiv_{r_{j'+2}} \phi_{j'+2}$. Then applying the induction hypothesis with $j = j'+1$ and $i = i'$, we have 
\begin{align}
[F_{\psi}^{(i')}]_{> \omega_{i'+1,j'+2}} = [F_{\phi_{j'+2}}^{(i')}]_{> \omega_{i'+1,j'+2}}
\label{appendix-psi-phi-1}
\end{align}
On the other hand, since $\psi \equiv_{r_{i'+1}} \phi_{i'+1}$, we can apply the induction hypothesis with $j = i'$ and $i = i'$ to obtain
\begin{align*}
[F_{\psi}^{(i')}]_{> \omega_{i'+1,i'+1}} = [F_{\phi_{i'+1}}^{(i')}]_{> \omega_{i'+1,i'+1}}
\end{align*}
Similarly, since $\phi_{j'+2} \equiv_{r_{i'+1}} \phi_{i'+1}$, we have
\begin{align*}
[F_{\phi_{j'+2}}^{(i')}]_{> \omega_{i'+1,i'+1}} = [F_{\phi_{i'+1}}^{(i')}]_{> \omega_{i'+1,i'+1}}
\end{align*}
Since $\omega_{i'+1,i'+1} = \omega_{i'+1}$, it follows that
\begin{align}
[F_{\psi}^{(i')}]_{> \omega_{i'+1}} 
	= [F_{\phi_{j'+2}}^{(i')}]_{> \omega_{i'+1}} 
	= [F_{\phi_{i'+1}}^{(i')}]_{> \omega_{i'+1}} 
	\label{appendix-psi-phi-2}
\end{align}
Identities \eqref{appendix-psi-phi-1}, \eqref{appendix-psi-phi-2} and assertion \eqref{F-k+1-agreement} of Lemma \ref{comparison-lemma-1} imply that $\psi$ and $\phi_{j'+2}$ satisfy the hypothesis of Lemma \ref{F-phi-psi-inductive-agreement} with $\mu = \omega_{i'+1, j'+2}$ and $k = i'+1$. Lemma \ref{F-phi-psi-inductive-agreement} therefore implies that
$$[F_{\psi}^{(i'+1)}]_{> \mu'} 
= [F_{\phi_{j'+2}}^{(i'+1)}]_{> \mu'}.$$
where $\mu' = (p_{i'+1} - 1)\omega_{i'+1} + \omega_{i'+1, j'+2}$. It is straightforward to check using \eqref{omega-defn} that $\mu' = \omega_{i'+2,j'+2}$, as required to complete the induction.
\end{proof}

\begin{proof}[\bf Proof of Lemma \ref{comparison-lemma-2}]
Since $\omega_{k+1} = \omega_{k+1,k+1}$ and $F_{\phi_{k+1}}^{(k)} = F_{\phi_{k+1}}$, Lemma \ref{comparison-lemma-2} is simply a special case of Corollary \ref{stronger-phi-psi-agreement}. \end{proof}

\section{Proof of Lemma \ref{poly-positive-prop}} \label{polynomial-appendix}

We freely use the notations of Section \ref{polynomial-section}. 

\begin{proof}[\bf Proof of assertion \eqref{wp-assertion} of Lemma \ref{poly-positive-prop}]
Since $f_0 = x$ and each $f_j$, $1 \leq j \leq l+1$, is monic in $y$ with $\deg_y(f_j) = p_0 \cdots p_{j-1}$ (where $p_0 := 1$), it is straightforward to see that each polynomial $f \in \cc[x,y]$ can be represented as a finite sum of the form
\begin{align*}
f = \sum_{\beta \in \zz_{\geq 0}^{l+2}} a_\beta f_0^{\beta_0} \cdots f_{l+1}^{\beta_{l+1}}
\end{align*}
where for each $\beta = (\beta_0, \ldots, \beta_{l+1})$, we have $a_\beta \in \cc$ and $\beta_j < p_j$, $1 \leq j \leq l$. It suffices to show that 
\begin{align}
\delta(f) &= \max\{\delta(f_0^{\beta_0} \cdots f_{l+1}^{\beta_{l+1}}): c_\beta \neq 0 \}
\end{align}
We compute $\delta(f)$ via identity \eqref{phi-delta-defn}. Assertion \eqref{last-expansion} of Theorem \ref{key-thm} imply that 
\begin{align}
f_j|_{y = \tilde \phi_\delta(x,\xi)} 
	= \begin{cases}
		c^*_j x^{\omega_j/\omega_0} + \ldt				& \text{for}\ 0 \leq j \leq l,\\
		(c^*_{l+1}\xi + c_{l+1}) x^{\omega_{l+1}/\omega_0} + \ldt	& \text{for}\ j = l + 1,
	  \end{cases} \label{g-j-substitution}
\end{align}
where $c^*_j \in \cc^*$, $0 \leq j \leq l$, $c_{l+1} \in \cc$, and $\ldt$ denotes terms with lower degree in $x$. Let $d :=  \max\{\delta(f_0^{\beta_0} \cdots f_{l+1}^{\beta_{l+1}}): a_\beta \neq 0 \}$ and $\scrB:= \{\beta:a_\beta \neq 0,\ \delta(f_0^{\beta_0} \cdots f_{l+1}^{\beta_{l+1}}) = d \}$. It follows that  
\begin{align}
& f|_{y = \tilde \phi_\delta(x,\xi)} =  c(\xi) x^{d/\omega_0} + \ldt,\ \text{where} 
\label{c-xi-0} \\
& c(\xi) := \sum_{\beta \in \scrB} a_\beta (c^*_{l+1}\xi + c_{l+1}) ^{\beta_{l+1}}\prod_{j=0}^{l} (c^*_j)^{\beta_j}  \label{c-xi}
\end{align}
Now, assertion \eqref{assertion-H-1-2} of Lemma \ref{delta-omega-lemma} implies that for two distinct elements $\beta, \beta'$ of $\scrB$, $\beta_{l+1} \neq \beta'_{l+1}$. Identity \eqref{c-xi} then implies that $c(\xi) \neq 0$, so that \eqref{c-xi-0} implies that $\delta(f) = d$, as required to complete the proof of assertion \eqref{wp-assertion}. 
\end{proof}

For each $j$, $0 \leq j \leq l+1$, let $\Omega_j \subseteq \zz$ be the semigroup generated by $\omega_0, \ldots, \omega_j$; recall that for $j \geq 1$, condition \polsub{j} implies that $\Omega_{j-1} \subseteq \zz_{\geq 0}$ (Lemma \ref{polytivity-lemma}).

\begin{lemma} \label{semi-isomorphism}
Assume \polsub{l+1} holds. Fix $j$, $1 \leq j \leq l$. Let $\bar J_{j+1}$ be the ideal in $C_j$ generated by $H_2, \ldots, H_{j+1}$. Let $t$ be an indeterminate. Then
$$C_j/\bar J_{j+1} \cong \cc[\Omega_j] \cong \cc[t^{\omega_0}, \ldots, t^{\omega_j}],$$
via the mapping $x \mapsto t^{\omega_0}$ and $y_i \mapsto b_i t^{\omega_i}$, $1 \leq i \leq j$, for some $b_1, \ldots, b_j \in \cc^*$. 
\end{lemma}

\begin{proof}
We proceed by induction on $j$. For $j=1$, identity \eqref{H-defn} and assertions \eqref{zero} and \eqref{decreasing-omega} of Lemma \ref{f-k-lemma} imply that
\begin{align*}
C_1/\bar J_2 = \cc[x,y_1]/\langle y_{1}^{p_1} - c_{1,0} x^{q_1} \rangle \cong \cc[t^{p_1}, t^{q_1}],
\end{align*}
where $t$ is an indeterminate and the isomorphism maps $x \mapsto t^{p_1}$ and $y_1 \mapsto c_{1,0}^{1/p_1}t^{q_1}$, where $c_{1,0}^{1/p_1}$ is a $p_1$-th root of $c_{1,0} \in \cc^*$. Since $\omega_0 = p_1p_2 \cdots p_l$ and $\omega_1 = q_1p_2 \cdots p_l$, this proves the lemma for $j = 1$. Now assume that the lemma is true for $j-1$, $2 \leq j \leq l$, i.e.\ there exists an isomorphism 
\begin{align*}
C_{j-1}/\bar J_{j} \cong \cc[t^{\omega_0},\ldots, t^{\omega_{j-1}}],
\end{align*}
which maps $x \mapsto t^{\omega_0}$ and $y_i \mapsto b_it^{\omega_i}$, $1 \leq i \leq j-1$ for some $b_{1}, \cdots, b_{j-1} \in \cc^*$. It follows that 
\begin{align*}
C_{j}/\bar J_{j+1} 
&= C_{j-1}[y_{j}]/\langle \bar J_{j}, y_{j}^{p_j} - c_{j,0} x^{\beta^0_{j,0}} y_1^{\beta^0_{j,1}} \cdots y_{j-1}^{\beta^0_{j,j-1}} \rangle 
\cong \cc[t^{\omega_0},\ldots, t^{\omega_{j-1}},y_j]/ \langle y_{j}^{p_j} - \tilde c  t^{p_j\omega_j} \rangle
\end{align*}
for some $\tilde c \in \cc^*$ (the last isomorphism uses assertion \eqref{decreasing-omega} of Lemma \ref{f-k-lemma}). Since $p_j = \min\{\alpha \in \zz_{> 0}; \alpha\omega_j \in \zz \omega_0 + \cdots + \zz \omega_{j-1}\}$ (assertion \eqref{alpha=p} of Proposition \ref{essential-value-prop}), it follows that 
\begin{align*}
\cc[t^{\omega_0},\ldots, t^{\omega_{j-1}},y_j]/ \langle y_{j}^{p_j} - \tilde c  t^{p_j\omega_j} \rangle \cong \cc[t^{\omega_0},\ldots, t^{\omega_{j}}]
\end{align*}
via a map which sends $y_j \mapsto (\tilde c)^{1/p_j}t^{\omega_j}$ ($(\tilde c)^{1/p_j}$ being a $p_j$-th root of $\tilde c$), which completes the induction and proves the lemma.
\end{proof}

Let $z$ be an indeterminate and $\hat C_{l+1} := C_{l+1}[z] = \cc[z,x,y_1, \ldots, y_{l+1}]$. Let $\hat \omega$ be the weighted degree on $\hat C_{l+1}$ such that $\hat \omega(z) = 1$ and $\hat \omega|_{C_{l+1}} = \omega$. Equip $\hat C^{l+1}$ with the grading determined by $\hat \omega$. Let $S^\delta$ be as in assertion \eqref{wp-assertion} of Lemma \ref{poly-positive-prop} and $\hat \pi: \hat C_{l+1} \to S^\delta$ be the map which sends $z \mapsto (1)_1$, $x \mapsto (x)_{\omega_0}$, and $y_j \mapsto (f_j)_{\omega_j}$, $1 \leq j \leq l+1$. Assertion \eqref{wp-assertion} implies that $\hat \pi$ is a surjective homomorphism of graded rings. Let $I$ be the ideal generated by $(1)_1$ in $S^{\delta}$ and $\hat J_{l+1} := \hat \pi^{-1}(I) \subseteq \hat C_{l+1}$. 

\begin{claim} \label{generating-claim}
$\hat J_{l+1}$ is generated by $\hat \scrB_{l+1} := (H_{l+1}, \ldots, H_2, z)$. 
\end{claim}

\begin{proof}
Let $\bar J_{l+1}$ be the ideal of $C_l$ as defined in Lemma \ref{semi-isomorphism}, and $\hat J'_{l+1}$ be the ideal of $\hat C_{l+1}$ generated by $\bar J_{l+1}$ and $z$. It is straightforward to see that $\hat J'_{l+1} \subseteq \hat J_{l+1}$. Lemma \ref{semi-isomorphism} implies that 
\begin{align*}
\hat C_{l+1}/\hat J'_{l+1} \cong \cc[t^{\omega_0},\ldots, t^{\omega_l}, y_{l+1}]
\end{align*}
Let $R := \cc[t^{\omega_0},\ldots, t^{\omega_l}, y_{l+1}]$. Then $S^\delta/I \cong \hat C_{l+1}/\hat J_{l+1} \cong R/\ppp$ for some prime ideal $\ppp$ of $R$. Now, it follows from the construction of $S^\delta$ that $\dim(S^\delta) = 3$. Since $I$ is the principal ideal generated by a non-zero divisor in $S^\delta$, it follows that $\dim(R/\ppp) = \dim(S^\delta/I) = 2$. Since $R$ is an integral domain of dimension $2$, we must have $\ppp = 0$, which implies the claim.
\end{proof}

\begin{proof}[\bf Proof of assertion \eqref{grobner-assertion} of Lemma \ref{poly-positive-prop}]
Since $J_{l+1} = \hat J_{l+1} \cap C_{l+1}$, Claim \ref{generating-claim} shows that $\scrB_{l+1}$ generates $J_{l+1}$. Therefore, to show that $\scrB_k$ is a Gr\"obner basis of $J_k$ with respect to $\prec_k$, it suffices to show that running a step of {\em Buchberger's algorithm} with $\scrB_{l+1}$ as input leaves $\scrB_{l+1}$ unchanged. We follow Buchberger's algorithm as described in \cite[Section 2.7]{littlesheacox}, which consists of performing the following steps for each pair of $H_i, H_j \in \scrB_{l+1}$, $2 \leq i < j \leq l+1$:\\

\paragraph{\bf Step 1: Compute the {\em S-polynomial} $S(H_i, H_j)$ of $H_i$ and $H_j$.} The leading terms of $H_i$ and $H_j$ with respect to $\prec$ are respectively $\lt_{\prec}(H_i) = y_{i-1}^{p_{i-1}}$ and $\lt_{\prec}(H_j) = y_{j-1}^{p_{j-1}}$, so that the {\em S-polynomial} of $H_i$ and $H_j$ is
\begin{align*}
S(H_i, H_j) &:= y_{j-1}^{p_{j-1}}H_i - y_{i-1}^{p_{i-1}}H_j  \\
&= -\left(c_{i-1,0}x^{\beta^0_{i-1,0}} y_1^{\beta^0_{i-1,1}} \cdots y_{i-2}^{\beta^0_{i-1,i-2}}\right)y_{j-1}^{p_{j-1}} 
+ \left(c_{j-1,0}x^{\beta^0_{j-1,0}} y_1^{\beta^0_{j-1,1}} \cdots y_{j-2}^{\beta^0_{j-1,j-2}}\right)y_{i-1}^{p_{i-1}}.
\end{align*}

\paragraph{\bf Step 2: Divide $S(H_i, H_j)$ by $\scrB_k$ and if the remainder is non-zero, then adjoin it to $\scrB_{l+1}$.} Since $i < j$, the leading term of $S(H_i,H_j)$ is 
\begin{align*}
\lt_{\prec}\left(S(H_i, H_j)\right) &= -\left(c_{i-1,0}x^{\beta^0_{i-1,0}} y_1^{\beta^0_{i-1,1}} \cdots y_{i-2}^{\beta^0_{i-1,i-2}}\right)y_{j-1}^{p_{j-1}}.
\end{align*}
Since $\beta^0_{i-1,j'} < p_{j'}$ for all $j'$, $1 \leq j' \leq i-1$ (assertion \eqref{exponent-bound} of Lemma \ref{f-k-lemma}), it follows that $H_j$ is the only element of $\scrB_{l+1}$ such that $\lt_{\prec}(H_j)$ divides $\lt_{\prec}\left(S(H_i, H_j)\right)$. The remainder of the division of $S(H_i, H_j)$ by $H_j$ is
\begin{align*}
S_1 &:= S(H_i, H_j) + \left(c_{i-1,0}x^{\beta^0_{i-1,0}} y_1^{\beta^0_{i-1,1}} \cdots y_{i-2}^{\beta^0_{i-1,i-2}}\right)H_j
= \left(c_{j-1,0}x^{\beta^0_{j-1,0}} y_1^{\beta^0_{j-1,1}} \cdots y_{j-2}^{\beta^0_{j-1,j-2}}\right)H_i,
\end{align*}
so that the leading term of $S_1$ is 
\begin{align*}
\lt_{\prec}(S_1) &= \left(c_{j-1,0}x^{\beta^0_{j-1,0}} y_1^{\beta^0_{j-1,1}} \cdots y_{j-2}^{\beta^0_{j-1,j-2}}\right)y_{i-1}^{p_{i-1}}.
\end{align*}
It follows as in the case of $S(H_i,H_j)$ that $H_i$ is the only element of $\scrB_{l+1}$ whose leading term divides $\lt_{\prec}(S_1)$. Since $H_i$ divides $S_1$, the remainder of the division of $S_1$ by $H_i$ is zero, and it follows that the remainder of the division of $S(H_i, H_k)$ by $\scrB_k$ is zero. Consequently {\bf Step 2 concludes without changing $\scrB_{l+1}$}.\\

It follows from the preceding paragraphs that running one step of Buchberger's algorithm keeps $\scrB_{l+1}$ unchanged, and consequently $\scrB_{l+1}$ is a Gr\"obner basis of $J_{l+1}$ with respect to $\prec$ \cite[Theorem 2.7.2]{littlesheacox}. This completes the proof of assertion \eqref{grobner-assertion} of Lemma \ref{poly-positive-prop}.
\end{proof}

\printbibliography

\end{document}